\documentclass[12pt]{amsart}

\usepackage{amsmath, amssymb}
\usepackage{graphicx}
\usepackage[dvipsnames]{xcolor}
\usepackage{stmaryrd}

\newtheorem{theorem}{Theorem}[section]
\newtheorem{corollary}[theorem]{Corollary}
\newtheorem{lemma}[theorem]{Lemma}
\newtheorem{proposition}[theorem]{Proposition}
\newtheorem{conjecture}[theorem]{Conjecture}
\theoremstyle{remark}
\newtheorem{example}[theorem]{Example}
\newtheorem{definition}[theorem]{Definition}
\newtheorem{remark}[theorem]{Remark}

\newcommand{\CC}{\mathbb{C}}
\newcommand{\QQ}{\mathbb{Q}}
\newcommand{\ZZ}{\mathbb{Z}}

\newcommand{\ukappa}{\underline{\kappa}}
\newcommand{\ulambda}{\underline{\lambda}}
\newcommand{\umu}{\underline{\mu}}
\newcommand{\unu}{\underline{\nu}}
\newcommand{\urho}{\underline{\rho}}
\newcommand{\ualpha}{\underline{\alpha}}
\newcommand{\ubeta}{\underline{\beta}}
\newcommand{\ugamma}{\underline{\gamma}}
\newcommand{\ua}{\underline{a}}
\newcommand{\ub}{\underline{b}}

\newcommand{\Fdot}{F_{\bullet}}
\newcommand{\Edot}{E_{\bullet}}
\newcommand{\Gdot}{G_{\bullet}}
\newcommand{\Kdot}{K_{\bullet}}
\newcommand{\Ldot}{L_{\bullet}}
\newcommand{\Tdot}{T_{\bullet}}

\newcommand{\calFdot}{\mathcal{F}_{\bullet}}
\newcommand{\calEdot}{\mathcal{E}_{\bullet}}

\newcommand{\Gr}{\mbox{\it Gr\,}}
\newcommand{\Fl}{{\mathbb F}\ell}
\newcommand{\Gal}{\mbox{\rm Gal}}

\newcommand{\I}{\raisebox{-1pt}{\,\includegraphics{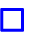}\,}}
\newcommand{\Ib}{\raisebox{-1pt}{\,\includegraphics{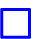}\,}}
\newcommand{\II}{\raisebox{-1pt}{\,\includegraphics{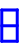}\,}}

\newcommand{\T}{\raisebox{-1pt}{\,\includegraphics{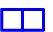}\,}}

\newcommand{\ThT}{\raisebox{-1pt}{\,\includegraphics{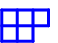}\,}}

\newcommand{\sI}{\raisebox{-1pt}{\includegraphics{pictures/1}}}

\newcommand{\sIII}{\raisebox{-1pt}{\includegraphics{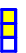}}}
\newcommand{\sncIII}{\raisebox{-1pt}{\includegraphics{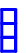}}}
\newcommand{\sTI}{\raisebox{-1pt}{\includegraphics{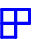}}}
\newcommand{\sTh}{\raisebox{-1pt}{\includegraphics{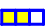}}}
\newcommand{\sncTh}{\raisebox{-1pt}{\includegraphics{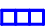}}}
\newcommand{\sIIII}{\raisebox{-1pt}{\includegraphics{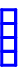}}}
\newcommand{\sTII}{\raisebox{-1pt}{\includegraphics{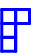}}}
\newcommand{\sTT}{\raisebox{-1pt}{\includegraphics{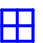}}}
\newcommand{\sThI}{\raisebox{-1pt}{\includegraphics{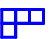}}}
\newcommand{\sF}{\raisebox{-1pt}{\includegraphics{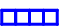}}}

\newcommand{\sFT}{\raisebox{-1pt}{\includegraphics{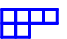}}}
\newcommand{\sThTh}{\raisebox{-1pt}{\includegraphics{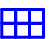}}}
\newcommand{\sTTII}{\raisebox{-1pt}{\includegraphics{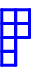}}}
\newcommand{\sTTT}{\raisebox{-1pt}{\includegraphics{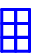}}}
\newcommand{\sFFTT}{\raisebox{-1pt}{\includegraphics{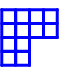}}}

\newcommand{\sFF}{\raisebox{-1pt}{\includegraphics{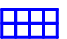}}}
\newcommand{\sFThI}{\raisebox{-1pt}{\includegraphics{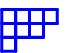}}}
\newcommand{\sFTT}{\raisebox{-1pt}{\includegraphics{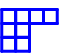}}}
\newcommand{\sThThII}{\raisebox{-1pt}{\includegraphics{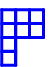}}}
\newcommand{\sThThT}{\raisebox{-1pt}{\includegraphics{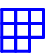}}}
\newcommand{\sFTII}{\raisebox{-1pt}{\includegraphics{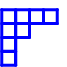}}}
\newcommand{\sTTTT}{\raisebox{-1pt}{\includegraphics{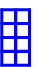}}}
\newcommand{\sThTTI}{\raisebox{-1pt}{\includegraphics{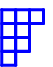}}}

\newcommand{\defcolor}[1]{{\color{blue}#1}}
\newcommand{\demph}[1]{\defcolor{{\sl #1}}}
\title{Galois Groups of Composed Schubert Problems} % running head version

\author[F.~Sottile]{Frank Sottile}     
\address{Frank Sottile\\     
         Department of Mathematics\\     
         Texas A\&M University\\     
         College Station\\     
         Texas \ 77843\\     
         USA}     
\email{sottile@math.tamu.edu}     
\urladdr{http://www.math.tamu.edu/\~{}sottile}
\author[R.~Williams]{Robert Williams}
\address{Robert Williams\\     
         Department of Mathematics\\     
         Sam Houston State University\\     
         Huntsville\\     
         Texas \ \\     
         USA}     
\email{rlw036@shsu.edu}     
\author[L.~Ying]{Li Ying}     
\address{Li Ying\\     
         Department of Mathematics\\     
         Texas A\&M University\\     
         College Station\\     
         Texas \ 77843\\     
         USA}     
\email{98yingli@gmail.com}     
\urladdr{http://www.math.tamu.edu/\~{}98yingli/}

\begin{document}
\begin{abstract}
 Two Schubert problems on possibly different Grassmannians may be composed to obtain a Schubert problem on a larger
 Grassmannian whose number of solutions is the product of the numbers of the original problems.
 This generalizes a construction discovered while classifying Schubert problems on the Grassmannian of $4$-planes in
 $\CC^9$ with imprimitive Galois groups.
 We give an  algebraic proof of the product formula.
 In a number of cases, we show that the Galois group of the composed Schubert problem is a subgroup of a wreath product of
 the Galois groups of the original problems, and is therefore imprimitive.
 We also present evidence for a conjecture that all composed Schubert problems have imprimitive Galois groups.
\end{abstract}

\maketitle

\begin{center}
 Dedicated to William Fulton on the occasion of his 80th birthday.
\end{center}

%%%%%%%%%%%%%%%%%%%%%%%%%%%%%%%%%%%%%%%%%%%%%%%%%%%%%%%%%%%%%%%%%%%%%%%%%%%%%%%%%
\section*{Introduction}
In his 1870 treatise ``Trait\'e des Substitutions et des \'equations alg\'ebrique'', C.\ Jordan~\cite{J1870} observed that
problems in enumerative geometry have Galois groups (of associated field extensions) which encode internal structure of
the enumerative problem.
Interest in such enumerative Galois groups was revived by Joe Harris through his seminal paper~\cite{Ha79} that
revisited and extended some of Jordan's work and gave a modern proof that the Galois group equals a monodromy group.
The Galois group of an enumerative problem is a subgroup of the symmetric group acting on its solutions.
Progress in understanding enumerative Galois groups was slow until the 2000's when new methods were introduced by Vakil and
others, some of which were particularly effective for Schubert problems on Grassmannians.

Vakil's geometric Littlewood-Richardson rule~\cite{Va06a} gave an algorithmic criterion that implies a Galois group of a
Schubert problem (a \demph{Schubert Galois group}) contains the alternating group (is \demph{at least alternating}).
This revealed some Schubert problems whose Galois groups might not be the full symmetric group, as well as many that were
at least alternating~\cite{Va06b}.
A numerical computation~\cite{LS09} of monodromy groups of some  Schubert problems found that all had full
symmetric Galois groups, providing evidence that the typical Schubert Galois group is full symmetric.
Vakil's experimentation suggested that Schubert Galois groups for Grassmannians of two- and three-dimensional
linear subspaces should always be full symmetric, but that this no longer holds for four-planes in $\CC^8$.

In fact (with help from Derksen), Vakil identified a Schubert problem with six solutions whose Galois group is the
symmetric group $S_4$, and the action is that of $S_4$ on subsets of $\{1,2,3,4\}$ of cardinality two.
It was later shown that all Schubert Galois groups for Grassmannians of 2-planes are at least alternating~\cite{BdCS} and
those for Grassmannians of 3-planes are doubly transitive~\cite{SW_double}, which supports 
Vakil's observations.
All Schubert problems on the next two Grassmannians of 4-planes whose Galois groups are not at least alternating were
classified, those in $\CC^8$ in~\cite{MSJ} and in $\CC^9$ in~\cite{GIVIX}.
Of over 35,000 nontrivial Schubert problems on these Grassmannians, only 163 had imprimitive Galois groups.
Besides the example of Derksen/Vakil, this study found two distinct types of constructions giving Schubert problems with 
imprimitive Galois groups.

Our goal is to better understand one of these types, called Type I in~\cite{GIVIX}.
We generalize the construction in~\cite{GIVIX} and call it \demph{composition of Schubert problems}
(Definition~\ref{Def:composition}). 
We prove a product formula (Theorem~\ref{Th:product}) for the number of solutions to a composition of Schubert problems,
which includes a combinatorial bijection involving sets of generalized Littlewood-Richardson tableaux
(explained in Remark~\ref{R:bijection}).
In many cases, we show that the Galois group of a composition of Schubert problems is imprimitive by embedding it into a
wreath product of Galois groups of Schubert problems on smaller Grassmannians.

In Section~\ref{S:SymmFn} we recall some standard facts about Schur functions and the Littlewood-Richardson formula, and
then deduce some results when the indexing partitions involve large rectangles.
In Section~\ref{S:composition} we use these formulas to give some results in the cohomology ring of Grassmannians.
We also define the main objects in this paper, a composable Schubert problem and a composition of Schubert problems,
and prove our main theorem which is a product formula for the number of solutions to a composition of Schubert problems.
In Section~\ref{Sec:GG} we define a family (\demph{block column Schubert problems}) of composable Schubert problems and
prove a geometric result that implies the Galois group of a composition with a block column Schubert problem is
imprimitive. 
We also present computational evidence for a conjecture that all nontrivial compositions have imprimitive Galois groups.

%%%%%%%%%%%%%%%%%%%%%%%%%%%%%%%%%%%%%%%%%%%%%%%%%%%%%%%%%%%%%%%%%%%%%%%%%%%%%%%%%
\section{Schur functions and Littlewood-Richardson coefficients}\label{S:SymmFn}

We recall facts about Schur functions and Littlewood-Richardson
coefficients which are found in~\cite{Fulton,Macdonald,Sagan,Stanley}.
We use these to establish some formulas for products of Schur functions whose partitions contain large rectangles.
This includes bijections involving generalized Littlewood-Richardson tableaux which count the coefficients in such
products.

%%%%%%%%%%%%%%%%%%%%%%%%%%%%%%%%%%%%%%%%%%%%%%%%%%%%%%%%%%%%%%%%%%%%%%%%%%%%%%%%%
\subsection{Symmetric functions}
A \demph{partition} \defcolor{$\lambda$} is a weakly decreasing sequence
$\lambda\colon \lambda_1\geq\lambda_2\geq\dotsb\geq\lambda_k>0$ of positive integers.
We set $\defcolor{|\lambda|}:=\lambda_1+\dotsb+\lambda_k$.
Here $k$ is the \demph{length} of $\lambda$, written \defcolor{$\ell(\lambda)$}.
Each integer $\lambda_i$ is a \demph{part} of $\lambda$.
When $a>\ell(\lambda)$, we write $\lambda_a=0$, and we identify partitions that differ only in trailing parts of size $0$.
The set of all partitions forms a partial order called \demph{Young's lattice} in which \defcolor{$\lambda\subseteq\mu$} when 
$\lambda_i\leq\mu_i$ for all $i$.
We often identify a partition with its \demph{Young diagram}, which is a left-justified array of $|\lambda|$ boxes with
$\lambda_i$ boxes in the $i$th row.
Here are three partitions and their Young diagrams.
\[
  (3)\ =\ \raisebox{-1.5pt}{\includegraphics{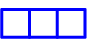}}
  \qquad
  (4,2)\ =\ \raisebox{-5.5pt}{\includegraphics{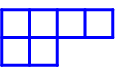}}
  \qquad
  (6,3,1)\ =\ \raisebox{-9.5pt}{\includegraphics{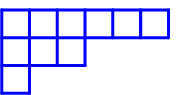}}
\]
The partial order among partitions in Young's lattice becomes set-theoretic containment of their Young diagrams, which
explains our notation $\lambda\subseteq\mu$.
The \demph{conjugate} \defcolor{$\lambda'$} of a partition $\lambda$ is the partition whose Young diagram is the
matrix-transpose of the Young diagram of $\lambda$.
For example,
\[
  (3)'\ =\ \raisebox{-1.5pt}{\includegraphics{pictures/3}}\,'\ =\
  \raisebox{-9.5pt}{\includegraphics{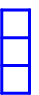}}
  \qquad\mbox{and}\qquad
  (4,3,1)'\ =\ \raisebox{-9.5pt}{\includegraphics{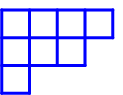}}
                \raisebox{9pt}{$\,'$}\ =\
  \raisebox{-13.5pt}{\includegraphics{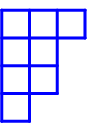}} 
\]
Conjugating a partition interchanges rows with columns.

A homogeneous formal power series $f\in\ZZ\llbracket x_1,x_2,\dotsc \rrbracket$ is a \demph{symmetric function} 
if it is invariant under permuting the variables.
For any nonnegative integer $n$, the set \defcolor{$\Lambda_n$} of symmetric functions of degree $n$ is a finite rank free
abelian group with a distinguished $\ZZ$-basis of \demph{Schur functions} \defcolor{$S_\lambda$} indexed by partitions
$\lambda$ with $|\lambda|=n$. 

The graded algebra \defcolor{$\Lambda$} of symmetric functions is the direct sum of all $\Lambda_n$ for $n\geq 0$.
Here, the empty partition \defcolor{$\emptyset$} is the unique partition of $0$, and $S_\emptyset=1$, the multiplicative
identity for $\Lambda$.
The set of all Schur functions forms a $\ZZ$-basis for $\Lambda$.
Consequently, there are integer \demph{Littlewood-Richardson coefficients} $\defcolor{c^\lambda_{\mu,\nu}}\in\ZZ$ defined
for triples $\lambda,\mu,\nu$ of partitions by expanding the product $S_\mu\cdot S_\nu$ in the
basis of Schur functions,
\[
  S_\mu\cdot S_\nu\ =\ \sum_{\lambda} c^\lambda_{\mu,\nu} S_\lambda\,.
\]
Conjugation of partitions induces the \demph{fundamental involution $\omega$} on $\Lambda$ defined by 
$\omega(S_{\lambda})=S_{\lambda'}$.
This is an algebra automorphism, and applying it to the product above shows that
$c^\lambda_{\mu,\nu}=c^{\lambda'}_{\mu',\nu'}$.
We will use this involution and identity throughout.

As $\Lambda$ is graded, if $c^\lambda_{\mu,\nu}\neq 0$, then $|\lambda|=|\mu|+|\nu|$.
More fundamentally, each $c^\lambda_{\mu,\nu}$ is nonnegative.
We say that a Schur function $S_\lambda$ \demph{occurs} in a product $S_{\mu^1}\dotsb S_{\mu^r}$ of Schur functions
if its coefficient in the expansion of that product in the Schur basis is positive.

The Littlewood-Richardson rule is a formula for the Littlewood-Richardson coefficients $c^\lambda_{\mu,\nu}$.
When $\nu\subset\lambda$, we write \defcolor{$\lambda/\nu$} for the set-theoretic difference of their diagrams, which is
called a \demph{skew shape}.
We may fill the boxes in $\lambda/\nu$ with positive integers that weakly increase left-to-right across each row and
strictly increase down each column, obtaining a \demph{skew tableau}.
The \demph{content} of a skew tableau $T$ is the sequence whose $i$th element is the number of occurrences of $i$ in $T$.
The word of a tableau $T$ is the sequence of its entries read  right-to-left, and
from the top row of $T$ to its bottom row.
A \demph{Littlewood-Richardson skew tableau} $T$ is a skew tableau in which the content of any initial segment of its
word is a partition.
(We extend the notion of content to any sequence of positive integers.)
That is, when reading $T$ in the order of its word, at every position within the word, the number of $i$'s encountered is
always at least the number of $(i{+}1)$'s encountered, for every $i$.
Below is the skew shape $(6,4,2,2)/(3,2)$ and two tableaux with that shape and content $(4,3,1,1)$ whose words are
$111322142$ and $111223142$, respectively.
Only the second is a Littlewood-Richardson tableau.
\[
  \begin{picture}(73,49)(-3.5,-3)
    \put(-3.5,-3){\includegraphics{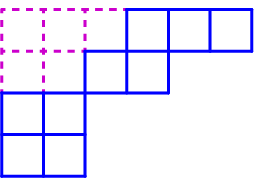}}
  \end{picture}
  \qquad
  \begin{picture}(73,49)(-3.5,-3)
    \put(-3.5,-3){\includegraphics{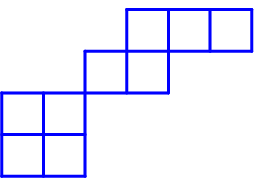}}
     \put(36,36){\small$1$} \put(48,36){\small$1$}\put(60,36){\small$1$}
     \put(24,24){\small$2$} \put(36,24){\small$3$}
     \put( 0,12){\small$1$} \put(12,12){\small$2$}
     \put( 0, 0){\small$2$} \put(12, 0){\small$4$}
  \end{picture}  
  \qquad
  \begin{picture}(73,49)(-3.5,-3)
    \put(-3.5,-3){\includegraphics{pictures/skewShape}}
     \put(36,36){\small$1$} \put(48,36){\small$1$}\put(60,36){\small$1$}
     \put(24,24){\small$2$} \put(36,24){\small$2$}
     \put( 0,12){\small$1$} \put(12,12){\small$3$}
     \put( 0, 0){\small$2$} \put(12, 0){\small$4$}
  \end{picture}  
\]
The Littlewood-Richardson coefficient  $c^\lambda_{\mu,\nu}$ is the number of Littlewood-Richardson tableaux of shape
$\lambda/\nu$ and content $\mu$.

We record some consequences of this formula.
First, if $c^\lambda_{\mu,\nu}\neq 0$, then $\mu,\nu\subseteq\lambda$ in Young's lattice.
The set of partitions $\lambda$ with $|\lambda|$ a fixed positive integer has a partial order, called \demph{dominance}.
For any positive integer $k$, write $|\lambda|_k$ for $\lambda_1+\dotsb+\lambda_k$.
Given partitions $\lambda$ and $\mu$ with $|\lambda|=|\mu|$, we say that $\mu$ \demph{dominates} $\lambda$,
written \defcolor{$\lambda \trianglelefteq \mu$} if, for all $k$, we have
$|\lambda|_k\leq|\mu|_k$.
Conjugation of partitions is an anti-involution of the dominance poset in that
$\lambda \trianglelefteq \mu$ if and only if $\mu' \trianglelefteq \lambda'$.

Given partitions $\mu$ and $\nu$, let \defcolor{$(\mu,\nu)$} be the partition of $|\mu|+|\nu|$ obtained by sorting the
sequence $(\mu_1,\dotsc,\mu_{\ell(\mu)},\nu_1,\dotsc,\nu_{\ell(\nu)})$ in decreasing order.
Let \defcolor{$\mu+\nu$} be the partition of  $|\mu|+|\nu|$ whose $k$th part is $\mu_k+\nu_k$.
Note that $(\mu,\nu)'=\mu'+\nu'$.

%%%%%%%%%%%%%%%%%%%%%%%%%%%%%%%%%%%%%%%%%%%%%%%%%%%%%%%%%%%%%%%%%%%%%%%%%%%%%%%%%
\begin{proposition}\label{P:LR_Dominance}
  If $\lambda$, $\mu$, and $\nu$ are partitions with $c^\lambda_{\mu,\nu}\neq 0$, then
\[
  (\mu,\nu)\ \trianglelefteq\ \lambda \trianglelefteq\ \mu+\nu\,.
\]
  Furthermore, $c^{(\mu,\nu)}_{\mu,\nu}=c^{\mu+\nu}_{\mu,\nu}=1$.
\end{proposition}
%%%%%%%%%%%%%%%%%%%%%%%%%%%%%%%%%%%%%%%%%%%%%%%%%%%%%%%%%%%%%%%%%%%%%%%%%%%%%%%%%

%%%%%%%%%%%%%%%%%%%%%%%%%%%%%%%%%%%%%%%%%%%%%%%%%%%%%%%%%%%%%%%%%%%%%%%%%%%%%%%%%
\begin{proof}
 In any Littlewood-Richardson skew tableau of shape $\lambda/\mu$ the first $i$ rows can only contain entries $1,\dotsc,i$,
 and thus if it has content $\nu$, there are at most $\nu_1+\dotsb+\nu_i$ boxes in these first $i$ rows.
 Since the $i$th row of $(\mu{+}\nu)/\mu$ has $\nu_i$ boxes, we see that there is a unique Littlewood-Richardson tableau of
 this shape and content $\nu$, and that for any other Littlewood-Richardson skew tableau of shape $\lambda/\mu$ and content
 $\nu$, we must have $\lambda\triangleleft\mu+\nu$.
 Applying conjugation gives the other inequality.
\end{proof}
%%%%%%%%%%%%%%%%%%%%%%%%%%%%%%%%%%%%%%%%%%%%%%%%%%%%%%%%%%%%%%%%%%%%%%%%%%%%%%%%%

%%%%%%%%%%%%%%%%%%%%%%%%%%%%%%%%%%%%%%%%%%%%%%%%%%%%%%%%%%%%%%%%%%%%%%%%%%%%%%%%%
\begin{corollary}\label{C:inequality}
  If $S_\lambda$ occurs in the product $S_\mu\cdot S_\nu$, then for any nonnegative integers $a,b$, we have
  $|\mu|_a+|\nu|_b\leq|\lambda|_{a+b}$.
\end{corollary}
%%%%%%%%%%%%%%%%%%%%%%%%%%%%%%%%%%%%%%%%%%%%%%%%%%%%%%%%%%%%%%%%%%%%%%%%%%%%%%%%%

%%%%%%%%%%%%%%%%%%%%%%%%%%%%%%%%%%%%%%%%%%%%%%%%%%%%%%%%%%%%%%%%%%%%%%%%%%%%%%%%%
\begin{proof}
  As $(\mu,\nu)$ is the decreasing rearrangement of the parts of $\mu$ and of $\nu$, we have
  $|\mu|_a+|\nu|_b\leq|(\mu,\nu)|_{a+b}$.
  Since $(\mu,\nu)\triangleleft\lambda$ by Proposition~\ref{P:LR_Dominance}, we have
  $|(\mu,\nu)|_{a+b}\leq|\lambda|_{a+b}$, which implies the inequality.
\end{proof}
%%%%%%%%%%%%%%%%%%%%%%%%%%%%%%%%%%%%%%%%%%%%%%%%%%%%%%%%%%%%%%%%%%%%%%%%%%%%%%%%%

%%%%%%%%%%%%%%%%%%%%%%%%%%%%%%%%%%%%%%%%%%%%%%%%%%%%%%%%%%%%%%%%%%%%%%%%%%%%%%%%%
\subsection{Schur functions of partitions with large rectangles}
We establish some results about products of Schur functions whose partitions contain large rectangles.
These are needed for our main combinatorial result about the number of solutions to a composition of Schubert problems in
Section~\ref{SS:composition}.

Let $a,d$ be nonnegative integers.
Write \defcolor{$\Box_{a,d}$} for the rectangular partition $(d,\dotsc,d)$ with $a$ parts, each of size $d$.
Its Young diagram is an $a\times d$ rectangular array of $ad$ boxes.
Suppose that $\mu,\alpha$ are partitions with $a\geq\ell(\mu)$ and $d\geq\alpha_1$.
Then \defcolor{$\Box_{a,d}{+}\mu$} is the partition whose $k$th part is $d+\mu_k$, for $k\leq a=\ell(\Box_{a,d}{+}\mu)$.
Also,  \defcolor{$(\Box_{a,d}{+}\mu,\alpha)$} is the partition whose parts are those of $\Box_{a,d}{+}\mu$ followed by
those of $\alpha$, as $d\geq\alpha_1$.
The Young diagram of $(\Box_{a,d}{+}\mu,\alpha)$ is obtained from the rectangle $\Box_{a,d}$ by placing the diagram of
$\mu$ to the right of the rectangle and the diagram of $\alpha$ below the rectangle.
\[
  \mu\ =\ \raisebox{-7.5pt}{\includegraphics{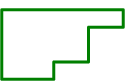}}
  \qquad
  \alpha\ =\ \raisebox{-5pt}{\includegraphics{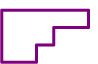}}
  \qquad \quad
  (\Box_{a,d}+\mu,\alpha)\ =\ 
  \raisebox{-17.56pt}{\begin{picture}(86,41)
    \put(0,0){\includegraphics{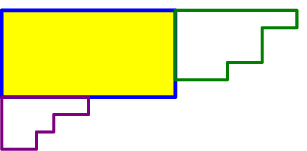}}
    \put(3,26){$a$}     \put(33,18){$d$}
    \put(3,7){$\alpha$}  \put(55,30){$\mu$}  
  \end{picture}}
\]
Observe that the conjugate of $(\Box_{a,d}+\mu,\alpha)$ is $(\Box_{d,a}+\alpha',\mu')$.

%%%%%%%%%%%%%%%%%%%%%%%%%%%%%%%%%%%%%%%%%%%%%%%%%%%%%%%%%%%%%%%%%%%%%%%%%%%%%%%%%
\begin{lemma}\label{Lemma:LiOne}
  Let $a,d$ be nonnegative integers and $\mu,\alpha$ be partitions with $a\geq\ell(\mu)$ and $d\geq \alpha_1$.
  Then
 \[
   S_{\Box_{a,d}+\mu}\cdot S_\alpha\ =\ S_{(\Box_{a,d}+\mu,\alpha)}\ +\
%   \sum_{\nu\triangleright (\Box_{a,d}+\mu,\alpha)} c^\nu_{\Box_{a,d}+\mu,\alpha} S_\nu\,,
   \sum_{(\Box_{a,d}+\mu,\alpha)\,\triangleleft\,\nu} c^\nu_{\Box_{a,d}+\mu,\alpha} S_\nu\,,
  \]
 and if $\nu$ indexes a nonzero term in the sum, then
  \begin{equation}\label{Eq:LiFormula} 
   |\nu|_a\ >\  |\Box_{a,d}+\mu|\ =\ ad+|\mu|\,.
  \end{equation}
\end{lemma}
%%%%%%%%%%%%%%%%%%%%%%%%%%%%%%%%%%%%%%%%%%%%%%%%%%%%%%%%%%%%%%%%%%%%%%%%%%%%%%%%%

%%%%%%%%%%%%%%%%%%%%%%%%%%%%%%%%%%%%%%%%%%%%%%%%%%%%%%%%%%%%%%%%%%%%%%%%%%%%%%%%%
\begin{proof}
  By Proposition~\ref{P:LR_Dominance}, the coefficient of
  $S_{(\Box_{a,d}+\mu,\alpha)}$ in the product is 1, and the other nonzero terms are indexed by partitions
  $\nu$ that  strictly dominate $(\Box_{a,d}+\mu,\alpha)$.
  Thus, 
 \[
  |\nu|_a\ \geq\ |(\Box_{a,d}+\mu,\alpha)|_a\ =\ |\Box_{a,d}+\mu|\ =\ ad + |\mu|\,.
 \]
 If the inequality is strict, then we have~\eqref{Eq:LiFormula}.

 Suppose that $|\nu|_a=ad + |\mu|$.
 Since $\nu\geq \Box_{a,d}{+}\mu$, and thus $\nu_i\geq d+\mu_i = (\Box_{a,d}{+}\mu)_i$, this implies that
 the partition, $(\nu_1,\dotsc,\nu_a)$, formed by the first $a$ parts of $\nu$ equals $\Box_{a,d}{+}\mu$.
 If we set $\beta=(\nu_{a+1},\dotsc)$, then $\nu=(\Box_{a,d}{+}\mu,\beta)$.
 In particular $\beta_1\leq d+\mu_a$ and $|\beta|=|\alpha|$.
 By the Littlewood-Richardson rule,  $c^\nu_{\Box_{a,d}+\mu,\alpha}$ counts the number of Littlewood-Richardson skew
 tableaux of shape $\nu/(\Box_{a,d}{+}\mu)=\beta$ and content $\alpha$.
 But a Littlewood-Richardson tableau of partition shape $\lambda$ must have content $\lambda$, which shows that
 $\beta=\alpha$ and completes the proof of the lemma.
\end{proof}
%%%%%%%%%%%%%%%%%%%%%%%%%%%%%%%%%%%%%%%%%%%%%%%%%%%%%%%%%%%%%%%%%%%%%%%%%%%%%%%%%

%%%%%%%%%%%%%%%%%%%%%%%%%%%%%%%%%%%%%%%%%%%%%%%%%%%%%%%%%%%%%%%%%%%%%%%%%%%%%%%%%
\begin{lemma}\label{L:KeyStep}
  Let $d,a_1,\dotsc,a_r$ be nonnegative integers and $\mu^1,\dotsc,\mu^r,\alpha^1,\dotsc,\alpha^r$ be partitions with 
  $a_i\geq\ell(\mu^i)$ and $d\geq\alpha^i_1$, for $i=1,\dotsc,r$.
  If $S_\nu$ occurs in the product  $\prod_{i=1}^rS_{(\Box_{a_i,d}+\mu^i,\alpha^i)}$
  % $S_{(d^{a_1}+\mu^1,\alpha^1)}\dotsb S_{(d^{a_r}+\mu^r,\alpha^r)}$
  of $r$ terms, then
 \[
    |\nu|_{a_1+\dotsb +a_r}\ \geq\ d(a_1+\dotsb +a_r)\, +\, |\mu^1| + \dotsb + |\mu^r|\,.
 \]
\end{lemma}
%%%%%%%%%%%%%%%%%%%%%%%%%%%%%%%%%%%%%%%%%%%%%%%%%%%%%%%%%%%%%%%%%%%%%%%%%%%%%%%%%

%%%%%%%%%%%%%%%%%%%%%%%%%%%%%%%%%%%%%%%%%%%%%%%%%%%%%%%%%%%%%%%%%%%%%%%%%%%%%%%%%
\begin{proof}
  We prove this by induction on the number $r$ of factors of the form $S_{(\Box_{a_i,d}+\mu^i,\alpha^i)}$.
  There is nothing to prove when $r=1$ as in this case $\nu=(\Box_{a,d}+\mu^1,\alpha^1)$.
  Suppose that the statement holds for $r{-}1$ factors, and that $S_\nu$ occurs in the product of
  $r$ factors.
  By the positivity of the coefficients of Schur functions in products of Schur functions, $S_\nu$ occurs in a product
  $S_{(\Box_{a_r,d}+\mu^r,\alpha^r)}\cdot S_\lambda$, where $S_\lambda$ occurs in the product with $r{-}1$ factors.
  By Corollary~\ref{C:inequality} and our induction hypothesis,
  \begin{eqnarray*}
    |\nu|_{a_1+\dotsb +a_r}& \geq& |(\Box_{a_r,d}{+}\mu^r,\alpha^r)|_{a_r}\,+\,|\lambda|_{a_1+\dotsb +a_{r-1}} \\
                        & \geq& d a_r \,+\, |\mu^r|\; +\; d(a_1+\dotsb +a_{r-1})\,+\, |\mu^1| + \dotsb + |\mu^{r-1}| \\
    &=& d(a_1+\dotsb +a_{r})\,+\, |\mu^1| + \dotsb + |\mu^{r}|\,,
  \end{eqnarray*}
  which completes the proof.
\end{proof}
%%%%%%%%%%%%%%%%%%%%%%%%%%%%%%%%%%%%%%%%%%%%%%%%%%%%%%%%%%%%%%%%%%%%%%%%%%%%%%%%%

We give a combinatorial description of some of the coefficients which appear in a product of the form
$\prod_{i=1}^rS_{\Box_{a_i,d}+\mu^i}$.
A \demph{generalized Littlewood-Richardson tableau}~\cite[Ex.~3.7]{BSS} \defcolor{$\Tdot$} of shape $\lambda$ and content 
$\defcolor{\umu}=(\mu^1,\dotsc,\mu^r)$ is a sequence of partitions,
\[
  \emptyset\ =\ \lambda^0\ \subseteq\ \lambda^1\ \subseteq\ \dotsb\ \subseteq\ \lambda^r\ =\ \lambda\,,
\]
 together with Littlewood-Richardson tableaux $T_1,\dotsc,T_r$ where $T_i$ has shape
$\lambda^i/\lambda^{i-1}$ and content $\mu^i$.
A repeated application of the Littlewood-Richardson rule implies the following.

%%%%%%%%%%%%%%%%%%%%%%%%%%%%%%%%%%%%%%%%%%%%%%%%%%%%%%%%%%%%%%%%%%%%%%%%%%%%%%%%%
\begin{proposition}\label{P:GLRTableaux}
   The coefficient \defcolor{$c^\lambda_{\umu}$} of $S_\lambda$ in the product $\prod_{i=1}^r S_{\mu^i}$ is equal to the
   number of generalized Littlewood-Richardson tableaux of shape $\lambda$ and content $\umu$.
\end{proposition}
%%%%%%%%%%%%%%%%%%%%%%%%%%%%%%%%%%%%%%%%%%%%%%%%%%%%%%%%%%%%%%%%%%%%%%%%%%%%%%%%%

Let $\lambda,\mu,\nu$ be partitions with $|\lambda|=|\mu|+|\nu|$ and $\mu,\nu\subset\lambda$.
Suppose that $l,m,n$ are integers with $l=m+n$ and $l\geq\ell(\lambda)$, $m\geq\ell(\mu)$, and $n\geq\ell(\nu)$.
For any $d\geq 0$ the skew shape $(\Box_{l,d}+\lambda)/(\Box_{n,d}+\nu)$ is obtained by placing
the rectangle $\Box_{m,d}$ to the left of the skew shape $\lambda/\nu$, with the last row of $\Box_{m,d}$ at row $l$.
Write \defcolor{$\Box_{m,d}+\lambda/\nu$} for this shape.
We illustrate this, shading the skew shapes.
\[
  \begin{picture}(91,41)(-40,0)
      \put(0,0){\includegraphics{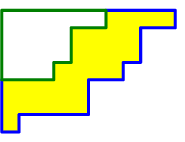}}
      \put(-40,18){$\lambda/\nu=$}
      \put(5,30){$\nu$}   \put(36,13){$\lambda$} 
  \end{picture}
   \qquad\qquad
  \begin{picture}(173,41)(-77,0)
    \put(0,0){\includegraphics{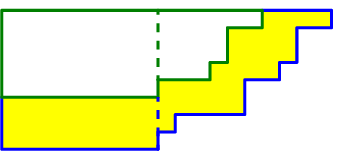}}
    \put(-77,18){$\Box_{m,d}+\lambda/\nu=$}
%    \put(-117, 1){$(\Box_{l,d}+\lambda)/(\Box_{n,d}+\nu)=$}
    \put(3,25){$n$}  \put(3,5){$m$} \put(20,18){$d$} 
  \end{picture}
\]
In any Littlewood-Richardson tableau $T$ of shape $\Box_{m,d}{+}\lambda/\nu$ with content $\Box_{m,d}{+}\mu$, the rectangle
$\Box_{m,d}$ in $T$ is \demph{frozen} in that each of its $d$ columns are filled with $1,\dotsc,m$.
%
%  The following needs a slight argument !
%
Removing the frozen rectangle gives a Littlewood-Richardson tableau of shape $\lambda/\nu$ and content $\mu$
(this may be proved by an induction on $d$, which we omit).

%%%%%%%%%%%%%%%%%%%%%%%%%%%%%%%%%%%%%%%%%%%%%%%%%%%%%%%%%%%%%%%%%%%%%%%%%%%%%%%%%
\begin{lemma}\label{L:frozen}
  Adding or removing frozen rectangles from skew Littlewood-Richardson tab\-leaux gives a bijection
\[
  \left\{\begin{minipage}[c]{2.2in}
      \begin{center}  Littlewood-Richardson tableaux of shape
        $\lambda/\nu$ and content $\mu$.
     \end{center}
     \end{minipage}\right\}
      \quad\longleftrightarrow\quad
  \left\{\begin{minipage}[c]{2.7in}
      \begin{center}   Littlewood-Richardson tableaux of shape
        $\Box_{m,d}{+}\lambda/\nu$ and content $\Box_{m,d}{+}\mu$.
     \end{center}
     \end{minipage}\right\}\,.
\]
\end{lemma}
%%%%%%%%%%%%%%%%%%%%%%%%%%%%%%%%%%%%%%%%%%%%%%%%%%%%%%%%%%%%%%%%%%%%%%%%%%%%%%%%%

Let $\umu:=(\mu^1,\dotsc,\mu^r)$ be a sequence of partitions and $\lambda$ a partition with
$|\mu^1|+\dotsb+|\mu^r|=|\lambda|$.
Let $\defcolor{\ua}:=(a_1,\dotsc,a_r)$ be a sequence of integers with $a_i\geq\ell(\mu^i)$.
Set $a:=|a_1|+\dotsb+|a_r|$.

%%%%%%%%%%%%%%%%%%%%%%%%%%%%%%%%%%%%%%%%%%%%%%%%%%%%%%%%%%%%%%%%%%%%%%%%%%%%%%%%%
\begin{corollary}\label{C:BoxesOff}
  For any $d\geq 0$, 
  the coefficient $c^\lambda_{\umu}$ of $S_\lambda$ in the product $\prod_i S_{\mu^i}$ is equal to the coefficient of
  $S_{\Box_{a,d}+\lambda}$ in the product $\prod_i S_{\Box_{a_i,d}+\mu^i}$.  
\end{corollary}
%%%%%%%%%%%%%%%%%%%%%%%%%%%%%%%%%%%%%%%%%%%%%%%%%%%%%%%%%%%%%%%%%%%%%%%%%%%%%%%%%

%%%%%%%%%%%%%%%%%%%%%%%%%%%%%%%%%%%%%%%%%%%%%%%%%%%%%%%%%%%%%%%%%%%%%%%%%%%%%%%%%
\begin{proof}
  Let \defcolor{$\Box_{\ua,d}+\umu$} be the sequence of partitions whose $i$th element is $\Box_{a_i,d}+\mu^i$.
  The bijection of Lemma~\ref{L:frozen} of adding or removing a frozen rectangle from Littlewood-Richardson tableaux extends
  to a bijection
  \begin{equation}\label{eq:BijGLRTableaux}
  \left\{\begin{minipage}[c]{1.9in}
      \begin{center}
          Generalized Littlewood- Richardson tableaux of shape $\lambda$ and content $\umu$.
     \end{center}
     \end{minipage}\right\}
     \ \longleftrightarrow\ 
  \left\{\begin{minipage}[c]{2.5in}
      \begin{center}
        Generalized Littlewood-Richardson tableaux of shape $\Box_{a,d}+\lambda$ and content $\Box_{\ua,d}+\umu$.
      \end{center}
     \end{minipage}\right\}\,.
 \end{equation}
  By Proposition~\ref{P:GLRTableaux}, these sets of tableaux have cardinality the two coefficients in the statement
  of the corollary, which completes the proof.
\end{proof}
%%%%%%%%%%%%%%%%%%%%%%%%%%%%%%%%%%%%%%%%%%%%%%%%%%%%%%%%%%%%%%%%%%%%%%%%%%%%%%%%%

%%%%%%%%%%%%%%%%%%%%%%%%%%%%%%%%%%%%%%%%%%%%%%%%%%%%%%%%%%%%%%%%%%%%%%%%%%%%%%%%%
\begin{example}
  Let us consider Corollary~\ref{C:BoxesOff} in the case when $|\lambda|=3$, $\umu=(\I,\I,\I)$, and $d=2$.
  Using the Pieri formula, we have
  \begin{equation}\label{Eq:degreeThree}
    S_{\sI}\cdot S_{\sI}\cdot S_{\sI}\ =\
    S_{\sncIII}\ +\ 2S_{\sTI}\ +\ S_{\sncTh}\,,
  \end{equation}
  and the generalized Littlewood-Richardson tableaux are the tableaux of content $(1,1,1)$,
  \[
    \begin{picture}(41,16)(-3.5,-15)
      \put(-3.5,-3){\includegraphics{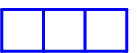}}
      \put( 0, 0){\small$1$} \put(12, 0){\small$2$}\put(24, 0){\small$3$}
    \end{picture}  
    \qquad
    \begin{picture}(41,16)(-3.5,-9)
      \put(-3.5,-3){\includegraphics{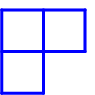}}
      \put( 0,12){\small$1$} \put(12,12){\small$2$}
      \put( 0, 0){\small$3$}
    \end{picture}  
    \qquad
    \begin{picture}(41,16)(-3.5,-9)
      \put(-3.5,-3){\includegraphics{pictures/21b}}
      \put( 0,12){\small$1$} \put(12,12){\small$3$}
      \put( 0, 0){\small$2$}
    \end{picture}  
    \qquad
    \begin{picture}(16,41)(-3.5,-3)
      \put(-3.5,-3){\includegraphics{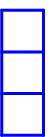}}
      \put(0,24){\small$1$}
      \put(0,12){\small$2$}
      \put(0, 0){\small$3$}
    \end{picture}  
   \]
   which shows~\eqref{Eq:degreeThree}.
   We let $\ua=(1,1,1)$, so that $\Box_{\ua,d}+\umu$ consists of three identical tableaux
   \raisebox{-2pt}{\,\includegraphics{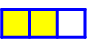}\,}.
   By the Pieri rule, a generalized Littlewood-Richardson tableau of shape $\kappa$ and content $\Box_{\ua,d}+\umu$
   is a tableau with content $(3,3,3)$.
   We show those of shape $\Box_{3,2}+\lambda$.
  \[
    \begin{picture}(53,40)(-3.5,-3)
      \put(-3.5,-3){\includegraphics{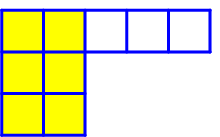}}
      \put(0,24){\small$1$} \put(12,24){\small$1$} \put(24,24){\small$1$} \put(36,24){\small$2$} \put(48,24){\small$3$}
      \put(0,12){\small$2$} \put(12,12){\small$2$} 
      \put(0, 0){\small$3$} \put(12, 0){\small$3$}    
    \end{picture}  
    \qquad
    \begin{picture}(53,40)(-3.5,-3)
      \put(-3.5,-3){\includegraphics{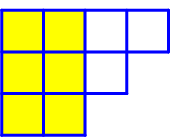}}
      \put(0,24){\small$1$}   \put(12,24){\small$1$}    \put(24,24){\small$1$}  \put(36,24){\small$2$}
      \put(0,12){\small$2$}   \put(12,12){\small$2$}    \put(24,12){\small$3$}
      \put(0, 0){\small$3$}   \put(12, 0){\small$3$}    
    \end{picture}  
    \qquad
    \begin{picture}(53,40)(-3.5,-3)
      \put(-3.5,-3){\includegraphics{pictures/432b}}
      \put(0,24){\small$1$}   \put(12,24){\small$1$}    \put(24,24){\small$1$}  \put(36,24){\small$3$}
      \put(0,12){\small$2$}   \put(12,12){\small$2$}    \put(24,12){\small$2$}
      \put(0, 0){\small$3$}   \put(12, 0){\small$3$}    
    \end{picture}  
    \qquad
    \begin{picture}(41,40)(-3.5,-3)
      \put(-3.5,-3){\includegraphics{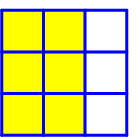}}
      \put(0,24){\small$1$}   \put(12,24){\small$1$}    \put(24,24){\small$1$}
      \put(0,12){\small$2$}   \put(12,12){\small$2$}    \put(24,12){\small$2$}
      \put(0, 0){\small$3$}   \put(12, 0){\small$3$}    \put(24, 0){\small$3$}
    \end{picture}  
  \]
  The `frozen rectangles' are shaded to better illustrate the bijection of Corollary~\ref{C:BoxesOff}.  
  \hfill$\diamond$
\end{example}
%%%%%%%%%%%%%%%%%%%%%%%%%%%%%%%%%%%%%%%%%%%%%%%%%%%%%%%%%%%%%%%%%%%%%%%%%%%%%%%%%

%%%%%%%%%%%%%%%%%%%%%%%%%%%%%%%%%%%%%%%%%%%%%%%%%%%%%%%%%%%%%%%%%%%%%%%%%%%%%%%%%
\section{Composition of Schubert problems}\label{S:composition}

We define Schubert problems on Grassmannians and use results from Section~\ref{S:SymmFn} to establish some formulas in the
cohomology ring of a Grassmannian.
We next define the composition of two Schubert problems, illustrating this notion with several examples.
We then formulate and prove our main theorem about the number of solutions to a composition of Schubert problems.
We use facts about the Grassmannian and its Schubert varieties and cohomology ring, as may be found
in~\cite{Fulton}.

%%%%%%%%%%%%%%%%%%%%%%%%%%%%%%%%%%%%%%%%%%%%%%%%%%%%%%%%%%%%%%%%%%%%%%%%%%%%%%%%%
\subsection{Schubert problems on Grassmannians}
For positive integers $a,b$, let \defcolor{$G(a,b)$} be the Grassmannian of $a$-dimensional linear subspaces of
$\CC^{a+b}$.
This is a smooth complex algebraic variety of dimension $ab$.
We will also need the alternative notation \defcolor{$\Gr(a,V)$} of $a$-dimensional linear subspaces of a complex vector
space $V$.
Then $\Gr(a,\CC^{a+b})=G(a,b)$.

Let $\lambda\subseteq\Box_{a,b}$ be a partition so that $\ell(\lambda)\leq a$ and $\lambda_1\leq b$.
A complete flag \defcolor{$\Fdot$} in $\CC^{a+b}$ is a sequence of linear subspaces
$\Fdot\colon F_1\subset F_2\subset\dotsb\subset F_{a+b}=\CC^{a+b}$ with $\dim F_i=i$.
Given a partition $\lambda\subseteq \Box_{a,b}$ and a flag $\Fdot$, we may define the \demph{Schubert variety} to be 
 \begin{equation}\label{Eq:SchubertVariety}
  \defcolor{\Omega_\lambda\Fdot}\ :=\
  \{ H\in G(a,b)\mid \dim H\cap F_{b+i-\lambda_i}\geq i\ \mbox{ for }i=1,\dotsc,a\}\,.
 \end{equation}
This has codimension $|\lambda|$ in $G(a,b)$.

A \demph{Schubert problem} on $G(a,b)$ is a list
$\defcolor{\ulambda}=(\lambda^1,\dotsc,\lambda^r)$ of partitions with $\lambda^i\subseteq\Box_{a,b}$ such that
$\defcolor{|\ulambda|}:=|\lambda^1|+\dotsb+|\lambda^r|=ab$.
We will call the partitions $\lambda^i$ the \demph{conditions} of the Schubert problem $\ulambda$.
An \demph{instance} of $\ulambda$ is determined by a list of flags $\defcolor{\calFdot}:=(\Fdot^1,\dotsc,\Fdot^r)$.
It consists of the points in the intersection, 
\begin{equation}\label{Eq:SchubertIntersection}
  \defcolor{\Omega_{\ulambda}\calFdot}\ :=\ 
  \Omega_{\lambda^1}\Fdot^1\,\cap\,
  \Omega_{\lambda^2}\Fdot^2\,\cap\,\dotsb\,\cap\,
  \Omega_{\lambda^r}\Fdot^r\,.
\end{equation}
When $\calFdot$ is general, this intersection is transverse~\cite{Kleiman}.
As $\ulambda$ is a Schubert problem, $\Omega_{\ulambda}\calFdot$ is zero-dimensional and therefore consists of finitely
many points. 
This number \defcolor{$\delta(\ulambda)$} of points does not depend upon the choice of flags.
It is the coefficient \defcolor{$c^{\Box_{a,b}}_{\ulambda}$} of the Schur function $S_{\Box_{a,b}}$ in the product
$S_{\lambda^1}\dotsb S_{\lambda^r}$.
A Schubert problem $\ulambda$ is \demph{trivial} if $\delta(\ulambda)\leq 1$.

We express this in the cohomology ring of the Grassmannian $G(a,b)$.
Given a partition $\lambda\subseteq\Box_{a,b}$, the cohomology class associated to a Schubert variety $\Omega_\lambda\Fdot$ is
the  \demph{Schubert class} \defcolor{$\sigma_\lambda$}.
The class associated to an intersection of Schubert varieties given by general flags is the product of the associated
Schubert classes, and \defcolor{$\sigma_{\Box_{a,b}}$} is the class of a point.
The map defined on the Schur basis by 
\[
  S_\lambda\ \longmapsto
  \left\{ \begin{array}{rcl} \sigma_\lambda&\ &\mbox{if } \lambda\subseteq \Box_{a,b}\\
            0             && \mbox{otherwise} \end{array}\right.
\]
is a homomorphism from $\Lambda$ onto this cohomology ring.
Thus we may evaluate a product in cohomology by applying this map to the corresponding product in $\Lambda$,
which removes all terms whose Schur function is indexed by a partition $\lambda$ that is not contained in the $a\times b$
rectangle, $\Box_{a,b}$.
Thus, for a Schubert problem $\ulambda=(\lambda^1,\dotsc,\lambda^r)$ and general flags $\calFdot$, the class of the
intersection $\Omega_{\ulambda}\calFdot$~\eqref{Eq:SchubertIntersection} is
\[
  \sigma_{\lambda^1} \sigma_{\lambda^2}\dotsb \sigma_{\lambda^r}
  \ =\ c^{\Box_{a,b}}_{\ulambda}\, \sigma_{\Box_{a,b}}
  \ =\ \delta(\ulambda)\, \sigma_{\Box_{a,b}}\,.
\]

The fundamental involution has a geometric counterpart.
The map sending an $a$-dimensional subspace $H$ of $V\simeq\CC^{a+b}$ to its annihilator $H^\perp$ in the dual space $V^*$
is an isomorphism $\Gr(a,V)\xrightarrow{\,\sim\,}\Gr(b,V^*)$.
Given a flag $\Fdot$ in $V$, its sequence of annihilators gives a flag $\Fdot^\perp$ in $V^*$ and for any partition
$\lambda\subseteq\Box_{a,b}$ the map $H\mapsto H^\perp$  sends $\Omega_\lambda\Fdot$ to $\Omega_{\lambda'}\Fdot^\perp$.
(Recall that $\lambda'$ is conjugate of $\lambda$.)
This induces an isomorphism on cohomology, sending the Schubert class $\sigma_\lambda$ for $G(a,b)$ to the class
$\sigma_{\lambda'}$ for $G(b,a)$.
Because of this and the fundamental involution on $\Lambda$, results we prove for a Schubert problem $\ulambda$ on one
Grassmannian $G(a,b)$ hold for its conjugate problem $\ulambda'$ on the dual Grassmannian $G(b,a)$.\smallskip

For  $\lambda\subseteq\Box_{a,b}$, let $\defcolor{\lambda^\vee}\subseteq\Box_{a,b}$ be the partition
whose $i$th part is $\lambda^\vee_i=b-\lambda_{a+1-i}$.
The skew diagram $\Box_{a,b}/\lambda^\vee$ is the rotation \defcolor{$\lambda^\circ$} of the diagram of $\lambda$ by
$180^\circ$.
When $a=5$, $b=6$, and $\lambda=(4,3,1,1)$, here are $\lambda$ and $\lambda^\vee$,
\[
  (4,3,1,1)\ =\ \raisebox{-15pt}{\includegraphics{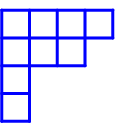}}\qquad\mbox{and}\qquad
  (4,3,1,1)^\vee\ =\ \raisebox{-19pt}{\includegraphics{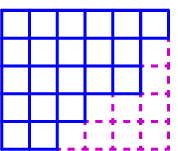}}\ =\ (6,5,5,3,2)\,.
\]

A key fact about the Schubert classes is Schubert's Duality Theorem~\cite[p.~149]{Fulton}, a version of which is the
following useful and well-known fact.

%%%%%%%%%%%%%%%%%%%%%%%%%%%%%%%%%%%%%%%%%%%%%%%%%%%%%%%%%%%%%%%%%%%%%%%%%%%%%%%%%
\begin{lemma}\label{L:zero}
  Let $\lambda,\mu\subseteq \Box_{a,b}$.
  Then $\sigma_\lambda\cdot\sigma_\mu\neq 0$ if and only if $\mu\subseteq\lambda^\vee$.
  Furthermore, $\sigma_\lambda\cdot\sigma_{\lambda^\vee}=\sigma_{\Box_{a,b}}$.
\end{lemma}
%%%%%%%%%%%%%%%%%%%%%%%%%%%%%%%%%%%%%%%%%%%%%%%%%%%%%%%%%%%%%%%%%%%%%%%%%%%%%%%%%

We give a consequence of Lemma~\ref{L:zero}, which is also well-known.

%%%%%%%%%%%%%%%%%%%%%%%%%%%%%%%%%%%%%%%%%%%%%%%%%%%%%%%%%%%%%%%%%%%%%%%%%%%%%%%%% 
\begin{corollary}\label{C:duality}
  The coefficient of $\sigma_\lambda$ in a cohomology class $\sigma$ equals the coefficient of $\sigma_{\Box_{a,b}}$ in
  the product $\sigma\cdot\sigma_{\lambda^\vee}$.
  In particular, $c^{\lambda^\vee}_{\mu,\nu}$ is the coefficient of $\sigma_{\Box_{a,b}}$ in
  the product $\sigma_\mu\cdot\sigma_\nu\cdot\sigma_{\lambda}$.
\end{corollary}
%%%%%%%%%%%%%%%%%%%%%%%%%%%%%%%%%%%%%%%%%%%%%%%%%%%%%%%%%%%%%%%%%%%%%%%%%%%%%%%%%

We deduce a lemma which constrains the shape of partitions indexing certain nonzero products.
This will be a key ingredient in our main theorem.
 
%%%%%%%%%%%%%%%%%%%%%%%%%%%%%%%%%%%%%%%%%%%%%%%%%%%%%%%%%%%%%%%%%%%%%%%%%%%%%%%%%
\begin{lemma}\label{L:FBigHook} 
 Suppose that $\kappa,\lambda\subseteq\Box_{a+c,b+d}$ are partitions with $\ell(\kappa)\leq a$ and $\lambda_1\leq b$.
 If $|\kappa|+|\lambda| > ad+ab+cb$, then $\sigma_\kappa\cdot\sigma_\lambda=0$ in the cohomology ring of $G(a{+}c,b{+}d)$.
 If $|\kappa|+|\lambda| = ad+ab+cb$ and  $\sigma_\kappa\cdot\sigma_\lambda\neq 0$, then
 $\sigma_\kappa\cdot\sigma_\lambda=\sigma_{\Box_{c,d}^\vee}$ and there is  partition $\alpha\subseteq\Box_{a,b}$ such that
 $\kappa=\Box_{a,d}+\alpha$ and $\lambda=(\Box_{c,b},\alpha^\vee)$. 
\end{lemma}
%%%%%%%%%%%%%%%%%%%%%%%%%%%%%%%%%%%%%%%%%%%%%%%%%%%%%%%%%%%%%%%%%%%%%%%%%%%%%%%%%

%%%%%%%%%%%%%%%%%%%%%%%%%%%%%%%%%%%%%%%%%%%%%%%%%%%%%%%%%%%%%%%%%%%%%%%%%%%%%%%%%
\begin{proof}
  Let $\kappa,\lambda\leq\Box_{a+c,b+d}$ be partitions with  $\ell(\kappa)\leq a$ and $\lambda_1\leq b$.
  By Lemma~\ref{L:zero}, $\sigma_\kappa\cdot\sigma_\lambda\neq 0$ if and only if $\kappa\subseteq\lambda^\vee$.
  Recall that $\Box_{a+c,b+d}/\lambda^\vee=\lambda^\circ$, which is the Young diagram of $\lambda$ rotated by $180^\circ$
  and placed in the southeast corner of $\Box_{a+c,b+d}$.
  Thus $\sigma_\kappa\cdot\sigma_\lambda\neq 0$ if and only if the diagrams of $\kappa$ and $\lambda^\circ$ are disjoint. 

  As $\ell(\kappa)\leq a$, the Young diagram of $\kappa$ lies in the first $a$ rows of $\Box_{a+c,b+d}$.
  Similarly, $\lambda^\circ$ lies in the last $b$ columns of  $\Box_{a+c,b+d}$, as $\lambda_1\leq b$.
  Since there are $ad+ab+cb$ boxes in these rows and columns, if
 \[
    |\kappa|+|\lambda|\ >\ ad+ab+cb\,,
 \]
  then the Young diagram of $\kappa$ must contain a box of $\lambda^\circ$ and $\sigma_\kappa\cdot\sigma_\lambda=0$.

  Suppose that $|\kappa|+|\lambda|=ad+ab+cb$ and $\sigma_\kappa\cdot\sigma_\lambda\neq 0$.
  Then $\kappa\subseteq\lambda^\vee$, and the disjoint union $\kappa\cup\lambda^\circ$ covers the first $a$ rows and last $b$
  columns of  $\Box_{a+c,b+d}$.
  In particular, this implies that $\kappa$ contains the northwest rectangle $\Box_{a,d}$ and $\lambda^\circ$ contains the
  southeast rectangle $\Box_{c,b}$.
  Thus there are partitions $\alpha,\beta\subseteq\Box_{a,b}$ such that $\kappa=\Box_{a,d}+\alpha$ and
  $\lambda=(\Box_{c,b},\beta)$. 
  As $\kappa\cup\lambda^\circ$ covers the northeast rectangle $\Box_{a,b}$, we see that $\beta=\alpha^\vee$.
  \[
    \begin{picture}(341,62)
      \put(10,0){\includegraphics{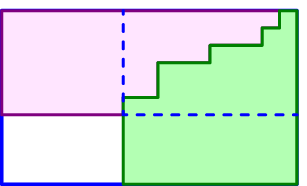}}
        \put( 0, 8){$c$}         \put( 0,33){$a$}
        \put(25,53){$d$}         \put(68,53){$b$}
        \put(17,7){$\Box_{c,d}$}  \put(65,7){$\lambda^\circ$}
        \put(22,33){$\kappa$}      
        
        \put(160,20){\includegraphics{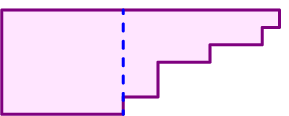}}
        \put(141,33){$\kappa:$}
        \put(169,33){$\Box_{a,d}$}   \put(205,40){$\alpha$}  
        
        \put(290,0){\includegraphics{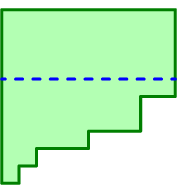}}
        \put(271,26){$\lambda:$}
        \put(300,15){$\alpha^\vee$}     \put(310,37){$\Box_{c,b}$}   

   \end{picture}
 \]

 Let $\rho\subseteq\Box_{a+c,b+d}$ be a partition such that $\sigma_\rho$ appears in the %nonzero
 product $\sigma_\kappa\cdot\sigma_\lambda$.
 Let $\nu:=\rho^\vee$.
 Then
 \[
   |\nu|\ =\ (a{+}c)(b{+}d)-|\kappa|-|\lambda|\ =\
   (a{+}c)(b{+}d)-ab-ad-cb\ =\ cd\,,
 \]
 and $\kappa,\nu,\lambda$ form a Schubert problem with
 $\sigma_\kappa\cdot\sigma_\nu\cdot\sigma_\lambda\neq 0$.
 Thus $\sigma_\kappa\cdot\sigma_\nu\neq 0$.
 By Proposition~\ref{P:LR_Dominance}, $\sigma_\kappa\cdot\sigma_\nu$ is a sum of terms $c^\gamma_{\kappa,\nu}\sigma_\gamma$ with
 $(\kappa,\nu)\trianglelefteq\gamma$.
 By Corollary~\ref{C:duality},
 $c^{\lambda^\vee}_{\kappa,\nu}\sigma_{\Box_{a+c,b+d}}= \sigma_\kappa\cdot\sigma_\nu\cdot\sigma_\lambda$.
 As the product is nonzero, $\sigma_{\lambda^\vee}$ occurs in $\sigma_\kappa\cdot\sigma_\nu$ and so we have
 $(\kappa,\nu)\trianglelefteq\lambda^\vee$, by Proposition~\ref{P:LR_Dominance}. 
 Since $\kappa_i=\lambda^\vee_i$ for $i\leq a$, an induction on $i$ shows that $(\kappa,\nu)_i=\kappa_i$ for  $i\leq a$.
 As $\lambda^\vee_{a+j}=d$ for $j=1,\dotsc,c$, another induction shows that $\nu_j=d$ for $j=1,\dotsc,c$.
 Consequently, $\nu=\Box_{c,d}$ and  $(\kappa,\nu)=\lambda^\vee$.
 By Proposition~\ref{P:LR_Dominance}, $c^{\lambda^\vee}_{\kappa,\nu}=c^{(\kappa,\nu)}_{\kappa,\nu}=1$ and thus
 $\sigma_\kappa\cdot\sigma_\nu\cdot\sigma_\lambda=\sigma_{\Box_{a+c,b+d}}$, by Corollary~\ref{C:duality}.

 Thus we have shown that if $|\nu|=cd$, so that $\lambda,\kappa,\nu$ is a Schubert problem, then
 \[
   \sigma_\kappa\cdot\sigma_\nu\cdot\sigma_\lambda \ =\
   \left\{\begin{array}{rcl} 1 &\ & \mbox{ if }\nu=\Box_{c,d}\\
                             0 &  & \mbox{ otherwise} \end{array}\right.\ .
 \]
 Thus by Corollary~\ref{C:duality}, $\sigma_\kappa\cdot\sigma_\lambda=\sigma_{\Box_{c,d}^\vee}$. 
\end{proof}
%%%%%%%%%%%%%%%%%%%%%%%%%%%%%%%%%%%%%%%%%%%%%%%%%%%%%%%%%%%%%%%%%%%%%%%%%%%%%%%%%

For any partition $\lambda$ and positive integer $b$, note that $|\lambda'|_b$ is the number of boxes in the first $b$
columns of the Young diagram of $\lambda$.

%%%%%%%%%%%%%%%%%%%%%%%%%%%%%%%%%%%%%%%%%%%%%%%%%%%%%%%%%%%%%%%%%%%%%%%%%%%%%%%%%
\begin{corollary}\label{C:moreZeroes}
  Let $\rho,\tau\subseteq\Box_{a+c,b+d}$ be partitions with $|\rho|_a+|\tau'|_b> ad+cb+ab$.
  Then $\sigma_\rho\cdot\sigma_\tau=0$ in the cohomology ring of $G(a{+}c,b{+}d)$.
\end{corollary}
%%%%%%%%%%%%%%%%%%%%%%%%%%%%%%%%%%%%%%%%%%%%%%%%%%%%%%%%%%%%%%%%%%%%%%%%%%%%%%%%%

%%%%%%%%%%%%%%%%%%%%%%%%%%%%%%%%%%%%%%%%%%%%%%%%%%%%%%%%%%%%%%%%%%%%%%%%%%%%%%%%%
\begin{proof}
 Let $\defcolor{\kappa}=(\rho_1,\dotsc,\rho_a)$ be the partition formed by the first $a$ parts of $\rho$ and
 $\alpha=(\rho_{a+1},\dotsc)$ be the partition formed by the remaining parts of $\rho$.
 Similarly, let $\defcolor{\lambda}$ be the partition formed from the first $b$ columns of the Young diagram of $\tau$ and
 $\beta$ the partition formed by the remaining columns of $\tau$.
 Then we have $\rho=(\kappa,\alpha)$ and $\tau=\lambda+\beta$.
 By Proposition~\ref{P:LR_Dominance}, $\sigma_\rho$ occurs in the product $\sigma_\kappa\cdot\sigma_\alpha$ and
 $\sigma_\tau$ occurs in the product $\sigma_\lambda\cdot\sigma_\beta$.

 Thus the coefficient of any Schubert class $\sigma_\gamma$ in the product $\sigma_\rho\cdot\sigma_\tau$ is at most the
 coefficient of  $\sigma_\gamma$ in the product
 $\sigma_\kappa\cdot\sigma_\alpha\cdot\sigma_\lambda\cdot\sigma_\beta=
 (\sigma_\kappa\cdot\sigma_\lambda)\cdot\sigma_\alpha\cdot\sigma_\beta$.
 As $\ell(\kappa)\leq a$ and $\lambda_1\leq b$ and  $|\kappa|+|\lambda| > ad+ab+cb$, Lemma~\ref{L:FBigHook}
 implies that 
 $\sigma_\kappa\cdot\sigma_\lambda=0$ in the cohomology ring of $G(a{+}c,b{+}d)$.
 This implies that $\sigma_\rho\cdot\sigma_\tau=0$.
\end{proof}
%%%%%%%%%%%%%%%%%%%%%%%%%%%%%%%%%%%%%%%%%%%%%%%%%%%%%%%%%%%%%%%%%%%%%%%%%%%%%%%%%

%%%%%%%%%%%%%%%%%%%%%%%%%%%%%%%%%%%%%%%%%%%%%%%%%%%%%%%%%%%%%%%%%%%%%%%%%%%%%%%%%
\subsection{Compositions of Schubert problems}\label{SS:composition}

A \demph{composable partition} of a Schubert problem $\ulambda$ on $G(a,b)$ is a partition
$\ulambda=(\umu,\unu)$ of the conditions of $\ulambda$, where 
$\umu=(\mu^1,\dotsc,\mu^r)$ and $\unu=(\nu^1,\dotsc,\nu^s)$, together with two sequences
$\defcolor{\ua}=(a_1,\dotsc,a_r)$ and $\defcolor{\ub}=(b_1,\dotsc,b_s)$ of nonnegative integers with
$a=|\ua|=a_1+\dotsb+a_r$ and $b=|\ub|=b_1+\dotsb+b_s$.
These data, $(\umu,\unu)$, $\ua$, and $\ub$, further satisfy
 \begin{equation}\label{Eq:ComposablePartition}
  a_i\ \geq\ \ell(\mu^i) \quad\mbox{for }i=1,\dotsc,r\qquad\mbox{and}\qquad
  b_j\ \geq\ \nu^j_1\quad\mbox{for }j=1,\dotsc,s\,.
 \end{equation}
We say that $\ulambda$ is \demph{composable} if it admits a composable partition.

%%%%%%%%%%%%%%%%%%%%%%%%%%%%%%%%%%%%%%%%%%%%%%%%%%%%%%%%%%%%%%%%%%%%%%%%%%%%%%%%%
\begin{example}\label{Ex:some_compositions}
  For example, $\ulambda=(\umu,\unu)$ with $\umu=((p),(q))$ and $\unu=(\I,\dotsc,\I)$ ($p{+}q$ occurrences of $\I$) is a
  composable Schubert problem on 
  $G(2,p{+}q)$ with $\ua=(1,1)$ and $\ub=(1,\dotsc,1)$   ($p{+}q$ occurrences of $1$).
  Also $\umu=(\T,\T,\I)$ and $\nu=(\II,\I,\I)$ with $\ua=\ub=(1,1,1)$ forms a composable partition of a Schubert problem on 
  $G(3,3)$,  as do $\umu=(\T,\T,\T)$ and $\nu=(\I,\I,\I)$, with the same $\ua$ and $\ub$. \hfill$\diamond$
\end{example}
%%%%%%%%%%%%%%%%%%%%%%%%%%%%%%%%%%%%%%%%%%%%%%%%%%%%%%%%%%%%%%%%%%%%%%%%%%%%%%%%%

%%%%%%%%%%%%%%%%%%%%%%%%%%%%%%%%%%%%%%%%%%%%%%%%%%%%%%%%%%%%%%%%%%%%%%%%%%%%%%%%%
\begin{remark}\label{R:expand}
  Not all Schubert problems are composable.
  For example, the Schubert problem $\ulambda=(\I,\I,\I,\I,\I,\I)$  on $G(2,3)$ is not composable as its number of
  conditions is $6$, which exceeds the sum $2+3$.
  However, it is geometrically equivalent to the Schubert problem $(\ThT,\I,\I,\I,\I,\I)$  on
  $G(2,5)$, which does have a composable partition, namely $\umu=(\ThT)$ and
  $\unu=(\I,\I,\I,\I,\I)$ with $\ua=(2)$ and $\ub=(1,1,1,1,1)$.  \hfill$\diamond$
\end{remark}
%%%%%%%%%%%%%%%%%%%%%%%%%%%%%%%%%%%%%%%%%%%%%%%%%%%%%%%%%%%%%%%%%%%%%%%%%%%%%%%%%

Geometrically equivalent means that every general instance of $(\ThT,\I,\I,\I,\I,\I)$ on $G(2,5)$ gives an 
instance of $(\I,\I,\I,\I,\I,\I)$  on $G(2,3)$ whose solutions correspond to solutions of the original problem, and every
general instance of the second problem occurs in this way.
This generalizes:
Every Schubert problem that is not composable is geometrically equivalent to a composable Schubert problem on a larger
Grassmannian.
A proof of this claim may be modeled on the construction of Remark~\ref{R:expand}.

%%%%%%%%%%%%%%%%%%%%%%%%%%%%%%%%%%%%%%%%%%%%%%%%%%%%%%%%%%%%%%%%%%%%%%%%%%%%%%%%%
\begin{definition}\label{Def:composition}
  Suppose that  $(\umu,\unu), \ua,\ub$ is a composable partition of a Schubert problem $\ulambda$ on $G(a,b)$,
  where  $\umu=(\mu^1,\dotsc,\mu^r)$ and $\unu=(\nu^1,\dotsc,\nu^s)$.
  Let $c,d$ be positive integers and let $\urho$ be any Schubert problem on $G(c,d)$.
  Let us partition the conditions of $\urho$ into three lists $\urho=(\ualpha,\ubeta,\ugamma)$, where
  $\ualpha=(\alpha^1,\dotsc,\alpha^r)$, $\ubeta=(\beta^1,\dotsc,\beta^s)$, and
  $\ugamma=(\gamma^1,\dotsc,\gamma^t)$.
  The \demph{composition} \defcolor{$\ulambda\circ\urho$} of these two Schubert problems is the list of partitions
  \[
    \bigl((\Box_{\ua,d}+\umu,\ualpha)\ ,\ (\Box_{c,\ub},\unu) + \ubeta\ ,\ \ugamma\,\bigr).
  \]
  Here, we define \defcolor{$(\Box_{\ua,d}{+}\umu,\ualpha)$} to be the sequence 
  $((\Box_{a_1,d}{+}\mu^1,\alpha^1),\dotsc, (\Box_{a_r,d}{+}\mu^r,\alpha^r) )$, and  
  similarly we define $(\Box_{c,\ub},\unu)+\ubeta$ to be the sequence
  $\bigl((\Box_{c,b_1},\nu^1)+\beta^1,\dotsc, (\Box_{c,b_s},\nu^s)+\beta^s\bigr)$.
  We always write a Schubert problem $\urho$ on $G(c,d)$ to compose with $\ulambda=(\umu,\unu)$ as a triple
  $(\ualpha,\ubeta,\ugamma)$ with $\umu$ and $\ualpha$ having the same number, $r$, of partitions, and both $\unu$ and
  $\ubeta$ have the same number, $s$, of partitions.
  We remark that any of the partitions $\mu^i$, $\nu^j$, $\alpha^i$, $\beta^j$, or $\gamma^k$ may be the zero partition
  $(0)$, whose Young diagram is $\emptyset$.\hfill$\diamond$   
\end{definition}
%%%%%%%%%%%%%%%%%%%%%%%%%%%%%%%%%%%%%%%%%%%%%%%%%%%%%%%%%%%%%%%%%%%%%%%%%%%%%%%%%

%%%%%%%%%%%%%%%%%%%%%%%%%%%%%%%%%%%%%%%%%%%%%%%%%%%%%%%%%%%%%%%%%%%%%%%%%%%%%%%%%
\begin{example}\label{Ex:compose}
  Let $\ulambda=(\umu,\unu)$ be the partition of a Schubert problem on $G(3,3)$ with  $\umu=(\T,\T,\I)$,
  $\nu=(\II,\I,\I)$, and $\ua=\ub=(1,1,1)$, which is from  Example~\ref{Ex:some_compositions}.  
  Let $\urho$ be the Schubert problem on $G(2,3)$ of Remark~\ref{R:expand}, expanded with some empty
  partitions, so that $\ualpha=(\I,\emptyset,\I)$, $\ubeta=(\I,\I,\emptyset)$, and $\ugamma=(\I,\I)$.
  The composition $\ulambda\circ\urho$ consists of the eight partitions
 \[
  \raisebox{-6pt}{\,\includegraphics{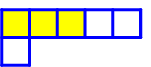}\,}\,,\,   
  \raisebox{-2pt}{\,\includegraphics{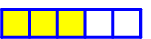}\,}\,,\,   
  \raisebox{-6pt}{\,\includegraphics{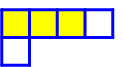}\,}\;\ ,\  \;  
  \raisebox{-14pt}{\,\includegraphics{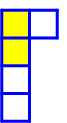}\,}\,,\,   
  \raisebox{-10pt}{\,\includegraphics{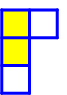}\,}\,,\,   
  \raisebox{-10pt}{\,\includegraphics{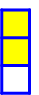}\,}\; \ ,\; \    
  \raisebox{-2pt}{\Ib}\, ,\,
  \raisebox{-2pt}{\Ib}   \ .
 \]
 We have shaded the rectangles $\Box_{a_i,3}$ and $\Box_{2,b_j}$ in the first six partitions.  \hfill$\diamond$
\end{example}
  
Our main combinatorial result is a product formula for the number of solutions to a composition of Schubert
problems. 

%%%%%%%%%%%%%%%%%%%%%%%%%%%%%%%%%%%%%%%%%%%%%%%%%%%%%%%%%%%%%%%%%%%%%%%%%%%%%%%%%
\begin{theorem}\label{Th:product}
  Let $(\umu,\unu), \ua,\ub$ be a composable partition of a Schubert problem $\ulambda$ on $G(a,b)$ and 
  $\urho=(\ualpha,\ubeta,\ugamma)$ a Schubert problem on $G(c,d)$.
  The number of solutions to $\ulambda\circ\urho$ is equal to the product of the number of solutions to $\ulambda$
  with the number of solutions to $\urho$,
  \[
      \delta(\ulambda\circ\urho)\ =\ 
      \delta(\ulambda)\cdot \delta(\urho)\,.
  \]
\end{theorem}
%%%%%%%%%%%%%%%%%%%%%%%%%%%%%%%%%%%%%%%%%%%%%%%%%%%%%%%%%%%%%%%%%%%%%%%%%%%%%%%%%

%%%%%%%%%%%%%%%%%%%%%%%%%%%%%%%%%%%%%%%%%%%%%%%%%%%%%%%%%%%%%%%%%%%%%%%%%%%%%%%%%
\begin{remark}\label{Derksen-Weyman}
  Derksen and Weyman use representations of quivers to prove a result~\cite[Thm.~7.14]{DW} expressing certain
  Littlewood-Richardson coefficients as products of other Littlewood-Richardson coefficients.
  When the Schubert problems $\ulambda$ and $\urho$  only involve three partitions each (e.g.\
  $\ulambda=(\mu^1,\mu^2,\nu^1)$ and $\urho=(\alpha^1,\alpha^2,\beta^1)$), then Theorem~\ref{Th:product} is a special case
  of this result of Derksen and Weyman.
  
  More generally, their methods lead to an alternative proof of
  Theorem~\ref{Th:product}. 
  First, their methods give an alternative proof of the enumerative consequence of Lemma~\ref{L:frozen} concerning the 
  equality of Littlewood-Richardson numbers,
  \[
    c^\lambda_{\mu,\nu}\ =\ c^{\Box_{l,d}+\lambda}_{\Box_{m,d}{+}\mu,\Box_{n,d}+\nu}\,,
  \]
  where $l=m+n$,  $l\geq\ell(\lambda)$, $m\geq\ell(\mu)$, and $n\geq\ell(\nu)$.
  Together with Lemmas~\ref{Lemma:LiOne} and~\ref{L:KeyStep} this implies Corollary~\ref{C:BoxesOff}, and eventually
  Theorem~\ref{Th:product}.
  We do not give the details of these arguments, as they require developing significant additional notation.
  We expect that representations of quivers may be used to give a direct proof
  of Theorem~\ref{Th:product}.  \hfill$\diamond$
\end{remark}
%%%%%%%%%%%%%%%%%%%%%%%%%%%%%%%%%%%%%%%%%%%%%%%%%%%%%%%%%%%%%%%%%%%%%%%%%%%%%%%%%

For an example of Theorem~\ref{Th:product}, observe that 
the Schubert problems $\ulambda$ and $\urho$ of Example~\ref{Ex:compose} both have five solutions, 
and thus $\ulambda\circ\urho$ has 25 solutions.
We first prove a preliminary result.

%%%%%%%%%%%%%%%%%%%%%%%%%%%%%%%%%%%%%%%%%%%%%%%%%%%%%%%%%%%%%%%%%%%%%%%%%%%%%%%%%
\begin{lemma}\label{L:BigHook}
  Let $(\umu,\unu), \ua,\ub$ be a composable partition of a Schubert problem $\ulambda$ on $G(a,b)$ and 
  let $c,d$ be nonnegative integers.
  In the cohomology ring of $G(a{+}c,b{+}d)$ we have
  \begin{equation}\label{Eq:boxyExpression}
   \prod_{i=1}^r \sigma_{\Box_{a_i,d}+\mu^i}\ \cdot\ 
   \prod_{j=1}^s \sigma_{(\Box_{c,b_j},\nu^j)}\ =\   \delta(\ulambda) \sigma_{\Box_{c,d}^\vee}\,,
  \end{equation}
 where $\Box_{c,d}^\vee$ is the complement of \/ $\Box_{c,d}$ in $\Box_{a+c,b+d}$.
\end{lemma}
%%%%%%%%%%%%%%%%%%%%%%%%%%%%%%%%%%%%%%%%%%%%%%%%%%%%%%%%%%%%%%%%%%%%%%%%%%%%%%%%%

Let $\lambda=(\I,\I,\I,\I)$ be a composable Schubert problem on $G(2,2)$ with $\delta(\lambda)=2$.
Consider applying Lemma~\ref{L:BigHook} to $\lambda$ with $c=d=2$.
Then the product~\eqref{Eq:boxyExpression} becomes,
 \begin{multline}\label{eq:FirstComposedExample}
   \quad \sigma_{\sTh}\cdot  \sigma_{\sTh} \cdot \sigma_{\sIII}\cdot  \sigma_{\sIII}
    \ =\ 
    \Bigl(\sigma_{\sFT} + \sigma_{\sThTh}\Bigr)\cdot   
    \Bigl(\sigma_{\sTTII} + \sigma_{\sTTT}\Bigr) \\
    \ =\ \sigma_{\sFT}\cdot \sigma_{\sTTII}\ +\
         \sigma_{\sFT}\cdot \sigma_{\sTTT}\ +\
        \sigma_{\sThTh}\cdot \sigma_{\sTTII}\ +\
        \sigma_{\sThTh}\cdot \sigma_{\sTTT}
    \ =\  2 \sigma_{\sFFTT}\ =\ 2 \sigma_{4422}\ ,\quad
   \end{multline}
as the first and the last of the products on the second line are zero, by Corollary~\ref{C:moreZeroes}, while both middle
ones are equal to the Schubert class $\sigma_{4422}$, by Lemma~\ref{L:FBigHook}.

%%%%%%%%%%%%%%%%%%%%%%%%%%%%%%%%%%%%%%%%%%%%%%%%%%%%%%%%%%%%%%%%%%%%%%%%%%%%%%%%%
\begin{proof}[Proof of Lemma~\ref{L:BigHook}]
 Let \defcolor{$\Box_{\ua,d}+\umu$} be the list of partitions appearing in the first product and
 \defcolor{$(\Box_{c,\ub},\unu)$} be the list of partitions appearing in the second product.
 Suppose that $\sigma_\kappa$ appears in the first product.
 An induction based on Proposition~\ref{P:LR_Dominance} implies that
 \[
   \ell(\kappa)\ \leq\ \sum_{i=1}^r \ell(\Box_{a_i,d}+\mu^i)\ =\ \sum_{i=1}^r a_i\ =\ a\,.
 \]
 Applying the fundamental involution $\omega$ shows that if $\sigma_\lambda$ occurs in the second product, then
 $\lambda_1\leq \sum_j b_j = b$.

 Observe that we have
 \begin{equation}\label{Eq:sizes}
   |\kappa|\ =\ |\Box_{\ua,d}+\umu|\ =\
   \sum_{i=1}^r |\Box_{a_i,d}+\mu^i|\ =\  \sum_{i=1}^r a_id\  + |\umu|\ =\ ad+|\umu|\,,
 \end{equation}
 and similarly $|\lambda|=bc+|\unu|$.
 Thus $|\kappa|+|\lambda|=ad+cb+ab$ as $|\umu|+|\unu|=ab$.
 By Lemma~\ref{L:FBigHook}, if $\sigma_\kappa\cdot\sigma_\lambda\neq 0$, then there exists a partition
 $\alpha\subseteq\Box_{a,b}$ for $G(a,b)$ such that $\kappa=\Box_{a,d}+\alpha$ and $\lambda=(\Box_{c,b},\alpha^\vee)$, and 
 $\sigma_\kappa\cdot\sigma_\lambda=\sigma_{\Box_{c,d}^\vee}$.
 Since~\eqref{Eq:sizes} holds, we see that $|\alpha|=|\umu|$ and $|\alpha^\vee|=|\unu|$.
 By Corollary~\ref{C:BoxesOff}, the coefficient of $\sigma_\kappa$ in 
 $\prod_i \sigma_{\Box_{a_i,d}+\mu^i}$ is $c^\alpha_{\umu}$.
 Applying the fundamental involution $\omega\colon\sigma_\rho\mapsto\sigma_{\rho'}$ shows that the coefficient
 of $\sigma_\lambda$ in  $\prod_j \sigma_{(\Box_{c,b_j},\nu^j)}$ is $c^{\alpha^\vee}_{\unu}$.

 Consequently, if we expand each product in the expression~\eqref{Eq:boxyExpression} as a sum of Schubert classes
 $\sigma_\rho$ and then expand the product of those sums, the only terms which contribute are
 $c^\alpha_{\umu}\cdot c^{\alpha^\vee}_{\unu}\cdot\sigma_{\Box_{a,d}+\alpha}\cdot\sigma_{(\Box_{c,b},\alpha^\vee)}$, for 
 $\alpha\subseteq\Box_{a,b}$ with $|\alpha|=|\umu|$.
 That is,
 \begin{multline*}
  \qquad
  \prod_{i=1}^r \sigma_{\Box_{a_i,d}+\mu^i}\ \cdot\ \prod_{j=1}^s \sigma_{(\Box_{c,b_j},\nu^j)}
  \ =\ \Bigl(\sum_\rho c^\rho_{\Box_{\ua,d}+\umu}\, \sigma_\rho\Bigr)\cdot
  \Bigl(\sum_\tau c^\tau_{(\Box_{c,\ub},\unu)}\, \sigma_\tau\Bigr)\\
  =\ \sum_{\substack{\alpha\subseteq\Box_{a,b}\\|\alpha|=|\umu|}} c^\alpha_{\umu}\cdot c^{\alpha^\vee}_{\unu}\cdot
       \sigma_{\Box_{a,d}+\alpha}\cdot\sigma_{(\Box_{c,b},\alpha^\vee)}
  =\ \sum_{\substack{\alpha\subseteq\Box_{a,b}\\|\alpha|=|\umu|}} c^\alpha_{\umu}\cdot c^{\alpha^\vee}_{\unu}\cdot
       \sigma_{\Box_{c,d}^\vee}\ ,\qquad
 \end{multline*}
 the last equality by Lemma~\ref{L:FBigHook}. 
A similar expansion of $\delta(\ulambda)\sigma_{\Box_{a,b}}=\prod_i \sigma_{\mu^i}\cdot\prod_j \sigma_{\nu^j}$ gives
\[
  \prod_{i=1}^r \sigma_{\mu^i}\ \cdot\ \prod_{j=1}^s \sigma_{\nu^j}
  \ =\ \Bigl(\sum_\rho c^\rho_{\umu}\, \sigma_\rho\Bigr)\cdot
  \Bigl(\sum_\tau c^\tau_{\unu}\, \sigma_\tau\Bigr)
  \ =\ \sum_{\substack{\alpha\subseteq\Box_{a,b}\\|\alpha|=|\umu|}} c^\alpha_{\umu}\cdot c^{\alpha^\vee}_{\unu}\cdot
      \sigma_{\Box_{a,b}}\ .
\]
 As this is $\delta(\ulambda)\sigma_{\Box_{a,b}}$, we conclude that the sum of the product of
 coefficients in both expressions equals $\delta(\ulambda)$, which completes the proof of the lemma.
 \end{proof}
%%%%%%%%%%%%%%%%%%%%%%%%%%%%%%%%%%%%%%%%%%%%%%%%%%%%%%%%%%%%%%%%%%%%%%%%%%%%%%%%%

%%%%%%%%%%%%%%%%%%%%%%%%%%%%%%%%%%%%%%%%%%%%%%%%%%%%%%%%%%%%%%%%%%%%%%%%%%%%%%%%%
\begin{example}\label{Ex:Running_composition} 
 The proof of Theorem~\ref{Th:product} involves expanding a product in cohomology in two different ways.
 Let us begin with an example, where $\ulambda=\urho=(\I,\I,\I,\I)$ and $a=b=c=d=2$.
 Then $\umu=\unu=(\I,\I)$ and $\ua=\ub=(1,1)$, and we have $\Box_{\ua,d}+\umu=(\sTh\,,\sTh)$ and
 $(\Box_{c,\ub},\unu)=(\raisebox{-3pt}{\sIII}\,,\raisebox{-3pt}{\sIII})$.
 The product that we expand is
 \begin{multline}\quad
   \sigma_{\sTh}\cdot  \sigma_{\sTh} \cdot \sigma_{\sIII}\cdot  \sigma_{\sIII}   \label{Eq:longProduct}
   \cdot  \sigma_{\sI}\cdot  \sigma_{\sI}\cdot  \sigma_{\sI}\cdot  \sigma_{\sI}\ =\\\
   2 \sigma_{\sFFTT} \cdot
   \bigl(2\sigma_{\sTT}\ +\ \sigma_{\sIIII} + 3\sigma_{\sTII} + 3\sigma_{\sThI} + \sigma_{\sF}  \bigr)
   \ =\ 2 \sigma_{\sFFTT} \cdot 2\sigma_{\sTT}\ =\ 4 \sigma_{\Box_{4,4}}\,. \quad
 \end{multline}
 The first equality uses~\eqref{eq:FirstComposedExample} to evaluate the product of the first four terms and expands
 the next four using the Pieri formula and~\eqref{Eq:degreeThree}.
 On the other hand,
 \[
   \sigma_{\sTh} \cdot \sigma_{\sI}\ =\  \sigma_{\sThI}\ +\ \sigma_{\sF}
     \qquad\mbox{and}\qquad
   \sigma_{\sIII} \cdot \sigma_{\sI}\ =\  \sigma_{\sTII}\ +\ \sigma_{\sIIII}\ .
 \]
 We may rearrange~\eqref{Eq:longProduct}  to get
 \[
    \bigl(\sigma_{\sTh} \cdot \sigma_{\sI}\bigr)^2\, \cdot\, 
    \bigl(\sigma_{\sIII} \cdot \sigma_{\sI}\bigr)^2\ =\
     \bigl(\sigma_{\sThI} + \sigma_{\sF}\bigr)^2\, \cdot\, 
     \bigl( \sigma_{\sTII} + \sigma_{\sIIII}\bigr)^2
     \ =\ \sigma_{\sThI}^2 \,\cdot\,\sigma_{\sTII}^2\,,
 \]
 as the other terms in the product vanish in the cohomology of $G(4,4)$.
 Since
 \begin{eqnarray*}
   \sigma_{\sThI}^2 &=&
          \sigma_{\sFTT}+\sigma_{\sThThII} + \sigma_{\sThThT}+\sigma_{\sFTII} \ +\ \sigma_{\sFF}+ 2\sigma_{\sFThI}\,,\ \mbox{and}\\
   \sigma_{\sTII}^2 &=&
          \sigma_{\sThThII}+\sigma_{\sFTT} + \sigma_{\sFTII} + \sigma_{\sThThT}\ +\ \sigma_{\sTTTT}+2\sigma_{\sThTTI}\,,\\
 \end{eqnarray*}
 if we use Lemma~\ref{L:zero} to expand $\sigma_{\sThI}^2 \cdot\sigma_{\sTII}^2\,$, the only non-zero products in the
 expansion are of the form $\sigma_{\lambda}\cdot\sigma_{\lambda^\vee}=\sigma_{\Box_{4,4}}$, where $\sigma_\lambda$ occurs
 in $\sigma_{\sThI}^2$ and $\sigma_{\lambda^\vee}$ occurs in $\sigma_{\sTII}^2$.
 The only such pairs are each of the first four terms in the two expressions above.
 Thus we again see that~\eqref{Eq:longProduct} equals $4\sigma_{\Box_{4,4}}$.  \hfill$\diamond$
\end{example}
%%%%%%%%%%%%%%%%%%%%%%%%%%%%%%%%%%%%%%%%%%%%%%%%%%%%%%%%%%%%%%%%%%%%%%%%%%%%%%%%% 

%%%%%%%%%%%%%%%%%%%%%%%%%%%%%%%%%%%%%%%%%%%%%%%%%%%%%%%%%%%%%%%%%%%%%%%%%%%%%%%%% 
\begin{proof}[Proof of Theorem~\ref{Th:product}]
  We compute the product corresponding to the Schubert problem
  $( \Box_{\ua,d}+\umu, (\Box_{c,\ub},\unu), \urho )$ in two different ways.
  Using Lemma~\ref{L:BigHook},  in the cohomology of $G(a{+}c,b{+}d)$ we have
  \begin{equation}\label{Eq:ProductToCompare}
    \prod_{i=1}^r \sigma_{\Box_{a_i,d}+\mu^i}\ \cdot\ 
    \prod_{j=1}^s \sigma_{(\Box_{c,b_j},\nu^j)}\ \cdot\ \prod_{\rho\in\urho} \sigma_{\rho}
    \ =\  \delta(\ulambda) \sigma_{\Box_{c,d}^\vee}
     \ \cdot\ \prod_{\rho\in\urho} \sigma_{\rho}\ .
  \end{equation}
  Since $\urho$ is a Schubert problem on $G(c,d)$, the only partition $\kappa\subseteq\Box_{c,d}$ with $\sigma_\kappa$ appearing
  in the last product is $\Box_{c,d}$ and the corresponding term is $\delta(\urho)\sigma_{\Box_{c,d}}$.
  By Lemma~\ref{L:zero}, the product in~\eqref{Eq:ProductToCompare} is
  $\delta(\ulambda)\cdot \delta(\urho)\cdot \sigma_{\Box_{a+c,b+d}}$.

  As $\urho=(\ualpha,\ubeta,\ugamma)$, we may rearrange the product in~\eqref{Eq:ProductToCompare} to obtain
  \begin{equation}\label{Eq:Rearrange}
     \Bigl(\prod_{i=1}^r \sigma_{\Box_{a_i,d}+\mu^i}\cdot \sigma_{\alpha^i}\Bigr)\ \cdot\ 
     \Bigl(\prod_{j=1}^s \sigma_{(\Box_{c,b_j},\nu^j)}\cdot\sigma_{\beta^j}\Bigr)\ \cdot\
     \prod_{k=1}^t \sigma_{\gamma^k}\ .
  \end{equation}
  By Lemma~\ref{Lemma:LiOne}, each term in the first product expands as
  \begin{equation}\label{Eq:oneTermExpand}
    \sigma_{(\Box_{a_i,d}+\mu^i,\alpha^i)}\ +\
     \sum_{\kappa} c^\kappa_{\Box_{a_i,d}+\mu^i,\alpha^i} \,\sigma_\kappa\ ,
  \end{equation}
  where if $\kappa$ indexes a nonzero term in the sum, then $|\kappa|_{a_i}>a_id+|\mu^i|$.
  
  Expanding the product of these expressions gives a sum of products of the form
  \begin{equation}\label{Eq:manyKappas}
    \Bigl(\prod_{i=1}^r   c^{\kappa^i}_{\Box_{a_i,d}+\mu^i,\alpha^i}\Bigr)\,
    \sigma_{\kappa^1} \sigma_{\kappa^2} \dotsb\sigma_{\kappa^r}\,,
  \end{equation}
  where for each $i$, $|\kappa^i|_{a_i}\geq a_id+|\mu^i|$, and we have equality only when
  $\kappa^i=(\Box_{a_i,d}+\mu^i,\alpha^i)$.
  Suppose that a Schubert class $\sigma_\rho$ occurs in the expansion of such a product~\eqref{Eq:manyKappas}.
  Then by Corollary~\ref{C:inequality}, $|\rho|_a\geq ad+|\umu|$.
  If $|\rho|_a=ad+|\mu|$, then $\kappa^i=(\Box_{a_i,d}+\mu^i,\alpha^i)$  for each $i$, and so
  $\sigma_\rho$ occurs in the product  $\prod_i\sigma_{(\Box_{a_i,d}+\mu^i,\alpha^i)}$ of the first terms
  from~\eqref{Eq:oneTermExpand} and the other terms do not contribute. 

  Applying the fundamental involution $\omega$, each term in the second product may be expanded to give
  \begin{equation}\label{Eq:nuBetas}
    \sigma_{(\Box_{c,b_j},\nu^j) + \beta^j}\ +\
     \sum_{\kappa} c^\kappa_{(\Box_{c,b_j},\nu^j),\beta^j} \sigma_\kappa\ ,
  \end{equation}
  where if  $\kappa$ indexes a term in the sum, then $|\kappa'|_{b_j}> cb_j+|\nu^j|$.
  The same reasoning shows that if $\sigma_\tau$ appears in the expansion of the second product in~\eqref{Eq:Rearrange},
  then $|\tau'|_b\geq cb+|\unu|$.
  Furthermore, if $|\tau'|_b= cb+|\unu|$, then $\sigma_{\tau}$ occurs in the product
  $\prod_j\sigma_{(\Box_{c,b_j},\nu^j) + \beta^j}$ of the first terms in~\eqref{Eq:nuBetas} and the other terms do not
  contribute. 
  
  Now suppose that $\sigma_\rho$ occurs in the expansion of the first product and $\sigma_\tau$ in the second.
  By Corollary~\ref{C:moreZeroes}, if $|\rho|_a+|\tau'|_b> ad+cb+ab = ad+cb+|\umu|+|\unu|$, then the product
  $\sigma_\rho\cdot\sigma_\tau=0$ in the cohomology ring of $G(a{+}c,b{+}d)$.
  Thus if $\sigma_\rho\cdot\sigma_\tau\neq 0$, then $|\rho|_a=ad+|\mu|$ and $|\tau'|_b= cb+|\unu|$.
  These arguments and observations imply the identity in the cohomology ring of $G(a{+}c,b{+}d)$,
\[
    \Bigl(\prod_{i=1}^r \sigma_{\Box_{a_i,d}+\mu^i}\cdot \sigma_{\alpha^i}\Bigr)\ \cdot\ 
    \Bigl(\prod_{j=1}^s \sigma_{(\Box_{c,b_j},\nu^j)}\cdot\sigma_{\beta^j}\Bigr)
    \ =\ 
    \prod_{i=1}^r \sigma_{(\Box_{a_i,d}+\mu^i,\alpha^i)}\ \cdot\ 
    \prod_{j=1}^s \sigma_{(\Box_{c,b_j},\nu^j)+\beta^j}\ .
\]
 Thus in the cohomology ring of  $G(a{+}c,b{+}d)$ the product~\eqref{Eq:ProductToCompare} equals
 \begin{equation}\label{Eq:composed_product}
   \prod_{i=1}^r \sigma_{(\Box_{a_i,d}+\mu^i,\alpha^i)}\ \cdot\ 
    \prod_{j=1}^s \sigma_{(\Box_{c,b_j},\nu^j)+\beta^j}\ \cdot\ 
     \prod_{k=1}^t \sigma_{\gamma^k}\  ,
 \end{equation}
 which is the product of Schubert classes in the composition $\ulambda\circ\urho$.
 Thus the two Schubert problems $\ukappa=( \Box_{\ua,d}+\umu, (\Box_{c,\ub},\unu), \urho )$~\eqref{Eq:ProductToCompare} and 
 $\ulambda\circ\urho$~\eqref{Eq:composed_product} have the same number of solutions, so that
 $\delta(\ulambda)\cdot\delta(\urho)=\delta(\ukappa)=\delta(\ulambda\circ\urho)$, which completes the proof. 
\end{proof}
%%%%%%%%%%%%%%%%%%%%%%%%%%%%%%%%%%%%%%%%%%%%%%%%%%%%%%%%%%%%%%%%%%%%%%%%%%%%%%%%%

%%%%%%%%%%%%%%%%%%%%%%%%%%%%%%%%%%%%%%%%%%%%%%%%%%%%%%%%%%%%%%%%%%%%%%%%%%%%%%%%%
\begin{remark}\label{R:bijection}
 Under the bijection of~\eqref{eq:BijGLRTableaux} involving frozen rectangles, and its conjugate
 version, we get a bijection between generalized Littlewood-Richardson tableaux of shape $\Box_{a+c,b+d}$ and content
 $\ulambda\circ\urho$ and pairs $(S,T)$ of generalized Littlewood-Richardson tableaux, where
 $S$ has shape $\Box_{a,b}$ and content $\ulambda$ and $T$ has shape $\Box_{c,d}$ and content $\urho$.
 One of us (Sottile) observed such a bijection in 2012, and this led eventually to the notion of composition of Schubert
 problem.
 Let us illustrate this in the product in Example~\ref{Ex:Running_composition}
 for the composed Schubert problem $\ulambda\circ\ulambda$, where $\ulambda=(\I,\I,\I,\I)$.

 We first shade the `frozen rectangles' in the partitions encoding terms of~\eqref{eq:FirstComposedExample}, filling in the
 remaining boxes with $1,2,3,4$, one number for each partition
 $(\sTh,\sTh,\raisebox{-3pt}{\sIII},\raisebox{-3pt}{\sIII})$ indexing the product, and omitting 
 terms that do not contribute.
 We write the product schematically as sums of indexing partitions.
 \[
   \left(\raisebox{-10pt}{
   \begin{picture}(38,28)(-3.5,-3)
     \put(-3.5,-3){\includegraphics{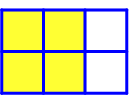}}
      \put(24, 12){\small$1$}  \put(24, 0){\small$2$} 
    \end{picture}
    \raisebox{10pt}{\ $+$\ }
    \begin{picture}(50,28)(-3.5,-3)
     \put(-3.5,-3){\includegraphics{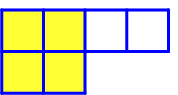}}
       \put(24, 12){\small$1$}  \put(36, 12){\small$2$} 
    \end{picture}
   }\right)
%     \raisebox{10pt}{\ $\cdot$\ }
    \cdot 
   \left(\raisebox{-22pt}{
   \begin{picture}(26,50)(-3.5,-3)
     \put(-3.5,-3){\includegraphics{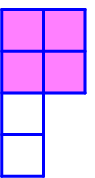}}
      \put(0, 12){\small$3$}  \put(0, 0){\small$4$}
    \end{picture}
    \raisebox{22pt}{\ $+$\ }
   \begin{picture}(28,50)(-3.5,-15)
     \put(-3.5,-3){\includegraphics{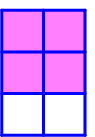}}
      \put(0, 0){\small$3$}  \put(12, 0){\small$4$}
    \end{picture}}\right)
   \ =\ 
   \raisebox{-22pt}{
     \begin{picture}(50,50)(-3.5,-3)
     \put(-3.5,-3){\includegraphics{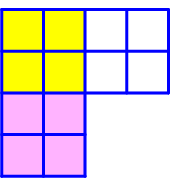}}
      \put(24, 36){\small$1$}  \put(36,36){\small$3$} 
      \put(24, 24){\small$2$}  \put(36,24){\small$4$} 
    \end{picture}
    \raisebox{22pt}{\ $+$\ }
   \begin{picture}(50,50)(-3.5,-3)
     \put(-3.5,-3){\includegraphics{pictures/4422b}}
      \put(24, 36){\small$1$}  \put(36,36){\small$2$} 
      \put(24, 24){\small$3$}  \put(36,24){\small$4$} 
    \end{picture}}
    \ .
  \]
  We similarly expand $\sigma_{\sI}^4$ using the letters $\alpha,\beta,\gamma,\delta$ instead of numbers
  to get
 \[
    \I\cdot \I\cdot \I\cdot \I\ =\ 
   \raisebox{-10pt}{
    \begin{picture}(25.5,26)(-3.5,-3)
     \put(-3.5,-3){\includegraphics{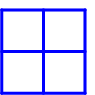}}
      \put(0, 12){\small$\alpha$}  \put(12,12){\small$\gamma$} 
      \put(0,  0){\small$\beta$}  \put(12, 0){\small$\delta$} 
    \end{picture}
   \raisebox{10pt}{\ $+$\ }
   \begin{picture}(25.5,26)(-3.5,-3)
     \put(-3.5,-3){\includegraphics{pictures/22b}}
      \put(0,12){\small$\alpha$}  \put(12,12){\small$\beta$} 
      \put(0, 0){\small$\gamma$}  \put(12, 0){\small$\delta$} 
    \end{picture}}
      \ .
  \]
  Similarly, if we expand the product of the eight terms coming from $\ulambda\circ\ulambda$, and express it schematically,
  we get
  \[
     \begin{picture}(51,52)(-3.5,-3)
     \put(-3.5,-3){\includegraphics{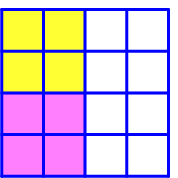}}
      \put(24, 36){\small$1$}        \put(36,36){\small$3$} 
      \put(24, 24){\small$2$}        \put(36,24){\small$4$}
      \put(24, 12){\small$\alpha$}   \put(36,12){\small$\gamma$} 
      \put(24,  0){\small$\beta$}    \put(36, 0){\small$\delta$} 
    \end{picture}
    \raisebox{22pt}{\quad $+$\quad }
   \begin{picture}(51,52)(-3.5,-3)
     \put(-3.5,-3){\includegraphics{pictures/4444b}}
      \put(24, 36){\small$1$}  \put(36,36){\small$2$} 
      \put(24, 24){\small$3$}  \put(36,24){\small$4$} 
      \put(24, 12){\small$\alpha$}   \put(36,12){\small$\gamma$} 
      \put(24,  0){\small$\beta$}    \put(36, 0){\small$\delta$} 
     \end{picture}
    \raisebox{22pt}{\quad $+$\quad }
     \begin{picture}(51,52)(-3.5,-3)
     \put(-3.5,-3){\includegraphics{pictures/4444b}}
      \put(24, 36){\small$1$}  \put(36,36){\small$3$} 
      \put(24, 24){\small$2$}  \put(36,24){\small$4$} 
      \put(24, 12){\small$\alpha$}   \put(36,12){\small$\beta$} 
      \put(24,  0){\small$\gamma$}   \put(36, 0){\small$\delta$} 
     \end{picture}
    \raisebox{22pt}{\quad $+$\quad }
   \begin{picture}(51,52)(-3.5,-3)
     \put(-3.5,-3){\includegraphics{pictures/4444b}}
      \put(24, 36){\small$1$}  \put(36,36){\small$2$} 
      \put(24, 24){\small$3$}  \put(36,24){\small$4$} 
      \put(24, 12){\small$\alpha$}   \put(36,12){\small$\beta$} 
      \put(24,  0){\small$\gamma$}   \put(36, 0){\small$\delta$} 
    \end{picture}
    \raisebox{22pt}{\ \ .}
    \eqno{ \raisebox{22pt}{$\diamond$}}
 \]
\end{remark} 
%%%%%%%%%%%%%%%%%%%%%%%%%%%%%%%%%%%%%%%%%%%%%%%%%%%%%%%%%%%%%%%%%%%%%%%%%%%%%%%%%

%%%%%%%%%%%%%%%%%%%%%%%%%%%%%%%%%%%%%%%%%%%%%%%%%%%%%%%%%%%%%%%%%%%%%%%%%%%%%%%%%
\section{Galois Groups of Composed Schubert Problems}
\label{Sec:GG}

Given a composable Schubert problem $\ulambda$ and any other Schubert problem $\urho$, Theorem~\ref{Th:product} shows that
$\delta(\ulambda\circ\urho)=\delta(\ulambda)\cdot\delta(\urho)$.
It was motivated by results in~\cite[Sec.~3.1]{GIVIX}, which we may now interpret as geometric proofs of
this product identity, for the  nontrivial composable Schubert problems $\ulambda$ on $G(2,2)$ and on $G(2,3)$.
A consequence of those geometric results is that the Galois group of such a composition is imprimitive,
specifically, it is a subgroup of a wreath product,
 \begin{equation}\label{Eq:GaloisWreath}
  \Gal_{\ulambda\circ\urho}\ \subset\
  \defcolor{\Gal_{\urho} \wr \Gal_{\ulambda}}\ :=\ ( \Gal_{\urho} )^{\delta(\ulambda)}\rtimes \Gal_{\ulambda}\ .
 \end{equation}
Further lemmas in~\cite[Sec.~3.1]{GIVIX} established equality for $\ulambda$ a composable Schubert problem on $G(2,2)$ or
$G(2,3)$.

We discuss Galois groups of Schubert problems, and following~\cite[Sec.~3]{GIVIX} identify a structure (a fibration of
Schubert problems) which implies the imprimitivity of a Galois group.
We then describe a class of Schubert problems all of whose compositions are fibered, and close with computational
evidence that more general compositions of Schubert problems have imprimitive Galois groups.

%%%%%%%%%%%%%%%%%%%%%%%%%%%%%%%%%%%%%%%%%%%%%%%%%%%%%%%%%%%%%%%%%%%%%%%%%%%%%%%%%
\subsection{Galois groups and fibrations of Schubert problems}
Let $\ulambda=(\lambda^1, \dots, \lambda^r)$ be a Schubert problem on $G(a,b)$ with $\delta(\ulambda)$ solutions.
We write \defcolor{$\Fl_{a+b}$} for the manifold of complete flags in $\CC^{a+b}$,
$\defcolor{\calFdot}:=(\Fdot^1,\dotsc,\Fdot^r)\in(\Fl_{a+b})^r$ for a list of flags, and 
\[
  \defcolor{\Omega_{\ulambda}\calFdot}\ :=\
    \Omega_{\lambda^1}\Fdot^1\,\cap\,
  \Omega_{\lambda^2}\Fdot^2\,\cap\,\dotsb\,\cap\,
  \Omega_{\lambda^r}\Fdot^r\,,
\]
for the instance~\eqref{Eq:SchubertIntersection} of the Schubert problem $\ulambda$ given by $\calFdot$.
Consider the total space of the Schubert problem $\ulambda$, 
\[
  \defcolor{\Omega_{\ulambda}}\ :=\
  \{(H,\calFdot)\in  G(a,b)\times(\Fl(a{+}b))^r
    \mid H \in \Omega_{\ulambda} \calFdot \}.
\]
This is irreducible, as the fiber of the projection $\Omega_{\ulambda}\to G(a,b)$ over a point $H\in G(a,b)$ is a product
$Z_{\lambda^1}H\times\dotsb\times Z_{\lambda^r}H$, where
\[
  Z_\lambda H\ :=\ \{\Fdot\mid H\in\Omega_\lambda\Fdot\}
\]
is a Schubert variety of the flag variety $\Fl(a{+}b)$.
Since each Schubert variety is irreducible, $\Omega_{\ulambda}$ as it is fibered over an irreducible variety $G(a,b)$ with
irreducible fibers.
  
The fiber $\pi^{-1}(\calFdot)$  of the projection $\pi\colon \Omega_{\ulambda} \rightarrow (\Fl(a{+}b))^r$ is the
instance $\Omega_{\ulambda}\calFdot$ of the Schubert problem $\ulambda$ given by the flags $\calFdot$. 
When the flags are general, this is a transverse intersection and consists of $\delta(\ulambda)$ points, and therefore
$\pi$ is a branched cover of degree $\delta(\ulambda)$.
The projection map $\pi$ induces an inclusion of function fields
$\pi^*\colon \CC \bigl( (\Fl_{a+b})^r \bigr) \hookrightarrow \CC(\Omega_{\ulambda})$, with the extension of degree
$\delta(\ulambda)$.
The \demph{Galois group} \defcolor{$\Gal_{\ulambda}$} of $\ulambda$ is the Galois group of the Galois closure of this
extension.

Such \demph{Schubert Galois groups} are beginning to be studied~\cite{BdCS,LS09,MSJ,GIVIX,SW_double,Va06b}, 
and while $\Gal_{\ulambda}$ is typically the full symmetric group $S_{\delta(\ulambda)}$, it is common
$\Gal_{\ulambda}$ to be imprimitive.
Let us recall some theory of permutation groups, for more, see~\cite{Wielandt}.
A permutation group is a subgroup $G$ of some symmetric group $S_\delta$, so that $G$ has a faithful
action on $\defcolor{[\delta]}:=\{1,\dotsc,\delta\}$.
The group $G$ is \demph{transitive} if it has a single orbit on $[\delta]$.
Galois groups are transitive permutation groups.
A \demph{block} of a permutation group $G$ is a subset $B$ of  $[\delta]$ such that for every $g\in G$ either
$gB\cap B=\emptyset$ or $gB=B$.
The orbit of a block under $G$ generates a partition of $[\delta]$ into blocks.
The group $G$ is primitive if its only blocks are singletons or $[\delta]$ itself, and \demph{imprimitive} otherwise.

When $G$ is imprimitive, there is a factorization $\delta=p\cdot q$ and an identification $[\delta]=[p]\times[q]$ with
$p,q>1$. 
If we let $\pi\colon[p]\times[q]\to [p]$ be the projection onto the first factor, then the fibers are blocks of $G$ and 
the action of $G$ on $[p]\times[q]$ preserves this fibration.
Conversely, if $G$ preserves such a nontrivial fibration, then it acts imprimitively.

A Schubert Galois group is imprimitive if its Schubert problem forms a fiber bundle
whose base and fibers are nontrivial Schubert problems in smaller Grassmannians.
Such a structure is called a decomposable projection in~\cite{AR}, which is equivalent to the Galois group being
imprimitive~\cite{SDSS,PS}.
The existence of such a structure implies that  $\Gal_{\ulambda}$ is a subgroup of the wreath product of the Galois groups
of the two smaller Schubert problems.
We give the definition of a fibration of Schubert problems from~\cite{GIVIX}.

%%%%%%%%%%%%%%%%%%%%%%%%%%%%%%%%%%%%%%%%%%%%%%%%%%%%%%%%%%%%%%%%%%%%%%%%%%%%%%%%%
\begin{definition}\label{D:fibration}
 Let $\ukappa$, $\ulambda$, and $\urho$ be nontrivial Schubert problems on $G(a{+}c,b{+}d)$, $G(a,b)$, and $G(c,d)$,
 respectively. 
 We say that \demph{$\ukappa$ is fibered over $\ulambda$ with fiber $\urho$} if the following hold.
 \begin{enumerate}
   \item For every general instance $\calFdot\in(\Fl_{a+c+b+d})^r$ of $\ukappa$, there is a subspace $V\subset\CC^{a+c+b+d}$ of
 dimension $a{+}b$ and an instance $\calEdot$ of $\ulambda$ in $\Gr(a,V)$ such that for every 
 $H\in\Omega_{\ukappa}\calFdot$, we have $H\cap V\in\Omega_{\ulambda}\calEdot$.
  \item If we set $W:=\CC^{a+c+b+d}/V$, then for any $h\in \Omega_{\ulambda}\calEdot$, there is an instance
 $\calFdot(h)$ of $\urho$ in $\Gr(c,W)$ such that if $\defcolor{h}:=H\cap V$, then $H/h\in\Gr(c,W)$ is a solution to
 $\Omega_{\urho}\calFdot(h)$.
  \item The association $H\mapsto (h,H/h)$, where $h:=H\cap V$, is a bijection between the sets of solutions 
    $\Omega_{\ukappa}\calFdot$ and $\Omega_{\ulambda}\calEdot\times \Omega_{\urho}\calFdot(h)$.
  \item For a given $V\simeq\CC^{a+b}$, all general instances $\calEdot$ of $\ulambda$ on $\Gr(a,V)$ may be obtained in
    (1). 
    For a given general instance $\calEdot$ of $\ulambda$ and $h\in\Omega_{\ulambda}\calEdot$, the instances $\calFdot(h)$
    of $\urho$ which arise are also general.
  \end{enumerate}
 This is called a fibration as the second instance $\calFdot(h)$ depends upon $h$.\hfill$\diamond$
\end{definition}
%%%%%%%%%%%%%%%%%%%%%%%%%%%%%%%%%%%%%%%%%%%%%%%%%%%%%%%%%%%%%%%%%%%%%%%%%%%%%%%%%

The following consequence of a fibration is Lemma~15 of \cite{GIVIX}.

%%%%%%%%%%%%%%%%%%%%%%%%%%%%%%%%%%%%%%%%%%%%%%%%%%%%%%%%%%%%%%%%%%%%%%%%%%%%%%%%%
\begin{lemma}\label{L:Fibered}
  If $\ukappa$ is a Schubert problem fibered over $\ulambda$ with fiber $\urho$, then
  $\delta(\ukappa) =\delta(\ulambda)\cdot \delta(\urho)$ and $\Gal_{\ukappa}$ is a subgroup of the wreath product
  $\Gal_{\ulambda} \wr \Gal_{\urho}$.
  If $\ulambda$ and $\urho$ are nontrivial, then $\Gal_{\ukappa}$ is imprimitive.
  Also, the image of $\Gal_{\ukappa}$ under the map to $\Gal_{\ulambda}$ is surjective, and the kernel, which is a
  subgroup of $(\Gal_{\urho})^{\delta(\ulambda)}$, has image $\Gal_{\urho}$ under projection to any factor.
\end{lemma}
%%%%%%%%%%%%%%%%%%%%%%%%%%%%%%%%%%%%%%%%%%%%%%%%%%%%%%%%%%%%%%%%%%%%%%%%%%%%%%%%% 

We make the following two conjectures.

%%%%%%%%%%%%%%%%%%%%%%%%%%%%%%%%%%%%%%%%%%%%%%%%%%%%%%%%%%%%%%%%%%%%%%%%%%%%%%%%%
\begin{conjecture}\label{C:imprimitive}
 The Galois group $\Gal_{\ulambda\circ\urho}$ of a composition of non-trivial Schubert problems is imprimitive.
\end{conjecture}
%%%%%%%%%%%%%%%%%%%%%%%%%%%%%%%%%%%%%%%%%%%%%%%%%%%%%%%%%%%%%%%%%%%%%%%%%%%%%%%%%

By Lemma~\ref{L:Fibered}, this is implied by a second, stronger conjecture.

%%%%%%%%%%%%%%%%%%%%%%%%%%%%%%%%%%%%%%%%%%%%%%%%%%%%%%%%%%%%%%%%%%%%%%%%%%%%%%%%%
\begin{conjecture}\label{C:fiber}
 A composed Schubert problem $\ulambda\circ\urho$ is fibered over $\ulambda$ with fiber $\urho$.
\end{conjecture}
%%%%%%%%%%%%%%%%%%%%%%%%%%%%%%%%%%%%%%%%%%%%%%%%%%%%%%%%%%%%%%%%%%%%%%%%%%%%%%%%%

We prove Conjecture~\ref{C:fiber} for a class of composable Schubert problems in the next section, and 
provide computational evidence for Conjecture~\ref{C:imprimitive} in Section~\ref{S:computation}.

%%%%%%%%%%%%%%%%%%%%%%%%%%%%%%%%%%%%%%%%%%%%%%%%%%%%%%%%%%%%%%%%%%%%%%%%%%%%%%%%% 
\subsection{Block column Schubert problems}

We identify a family of composable Schubert problems that we call block column Schubert problems. 
This includes all composable Schubert problems on $G(2,2)$ and $G(2,3)$ that were studied in~\cite[Sect.~3.1]{GIVIX}.
We show that for any block column Schubert problem $\ulambda$ and any Schubert problem
$\urho$, the composition $\ulambda\circ\urho$ is a Schubert problem that is fibered over $\ulambda$ with fiber $\rho$.
This proves that the Galois group of the composition $\ulambda\circ\urho$ is imprimitive when both $\ulambda$ and $\urho$
are nontrivial.  

%%%%%%%%%%%%%%%%%%%%%%%%%%%%%%%%%%%%%%%%%%%%%%%%%%%%%%%%%%%%%%%%%%%%%%%%%%%%%%%%%
\begin{definition}\label{D:BCSP}
 A \demph{block column Schubert problem} on $G(a,b)$ is a Schubert problem $\ulambda=(\umu,\unu)$ of the following form.
 The first part $\umu$ consists of two rectangular partitions, $\umu=(\Box_{a_1,b_1},\Box_{a_2,b_2})$, where $a=a_1+a_2$
 and $b=b_1+b_2$.
 The second part $\unu$ consists of two families of rectangular partitions,
 $\unu=(\Box_{a_2,m_1},\dotsc,\Box_{a_2,m_p}\,,\,\Box_{a_1,n_1},\dotsc,\Box_{a_1,n_q})$, where
 $m_1+\dotsb+m_p=b_1$ and $n_1+\dotsb+n_q=b_2$.\hfill$\diamond$
\end{definition}
%%%%%%%%%%%%%%%%%%%%%%%%%%%%%%%%%%%%%%%%%%%%%%%%%%%%%%%%%%%%%%%%%%%%%%%%%%%%%%%%%

%%%%%%%%%%%%%%%%%%%%%%%%%%%%%%%%%%%%%%%%%%%%%%%%%%%%%%%%%%%%%%%%%%%%%%%%%%%%%%%%%
\begin{lemma}
 A block column Schubert problem is composable.
\end{lemma}
%%%%%%%%%%%%%%%%%%%%%%%%%%%%%%%%%%%%%%%%%%%%%%%%%%%%%%%%%%%%%%%%%%%%%%%%%%%%%%%%%

%%%%%%%%%%%%%%%%%%%%%%%%%%%%%%%%%%%%%%%%%%%%%%%%%%%%%%%%%%%%%%%%%%%%%%%%%%%%%%%%%
\begin{proof}
  Let  $\ulambda=(\umu,\unu)$ be a block column Schubert problem.
  Let $\ua=(a_1,a_2)$ and $\ub=( m_1,\dotsc,m_p\,,\,n_1,\dotsc,n_q)$.
  Since these are the numbers of rows in the partitions of $\umu$ and columns in the partitions of $\unu$, respectively, we
  need only check their sums.
  Noting that $a_1+a_2=a$ and 
\[
    m_1+\dotsb+m_p \ +\ n_1+\dotsb+n_q\ =\ b_1 + b_2\ =\ b\,,
\]
  completes the proof.
\end{proof}
%%%%%%%%%%%%%%%%%%%%%%%%%%%%%%%%%%%%%%%%%%%%%%%%%%%%%%%%%%%%%%%%%%%%%%%%%%%%%%%%%

Let us consider some examples of block column Schubert problems.
One straightforward class is when $a_1=a_2=a$ and $b_1=b_2=b$ and $\ulambda=(\umu,\unu)$ with
$\umu=\unu=(\Box_{a,b},\Box_{a,b})$, which is a Schubert problem on $G(2a,2b)$.
This class includes the Schubert problem $(\I,\I\,,\,\I,\I)$ on $G(2,2)$.
Two more are
\[
  (\T,\T\,,\,\T,\T) \mbox{ on }G(2,4)
  \qquad\mbox{and}\qquad
   \bigl(\sTT,\sTT\,,\,\sTT,\sTT\bigr) \mbox{ on }G(4,4)\,.
 \]

Two others are $(\T,\I\,,\,\I,\I,\I)$ and $(\T,\I\,,\,\T,\I)$ on $G(2,3)$; these are two of the three
composable Schubert problems on $G(2,3)$.
The third, $(\sTI\,,\,\I,\I,\I)$, is geometrically equivalent to $(\I,\I\,,\,\I,\I)$  on $G(2,2)$, as is
$(\T,\I\,,\,\T,\I)$.  
A more interesting example is on $G(5,7)$ where $a_1=2$, $a_2=3$, $b_1=4$, and $b_2=3$, with
$\umu=(\Box_{2,4},\Box_{3,3})$ and $\unu=(\Box_{3,2},\Box_{3,1},\Box_{3,1}\,,\,\Box_{2,2},\Box_{2,1})$, which are the
following partitions
\[
  \raisebox{-8pt}{\includegraphics{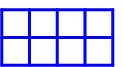}}\ ,\ 
  \raisebox{-12pt}{\includegraphics{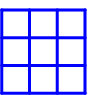}}
  \ \ ,\ \ 
   \raisebox{-12pt}{\includegraphics{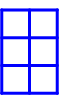}}\ ,\ 
   \raisebox{-12pt}{\includegraphics{pictures/111}}\ ,\ 
   \raisebox{-12pt}{\includegraphics{pictures/111}}
   \ \ ,\  \ 
   \raisebox{-8pt}{\includegraphics{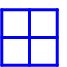}}\ ,\ 
   \raisebox{-8pt}{\includegraphics{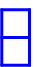}}\ .\ 
\]

Not all composable Schubert problems are block column.
While the first Schubert problem of Example~\ref{Ex:some_compositions},  $\umu=((p),(q))$ and $\unu=(\I,\dotsc,\I)$
($p{+}q$ occurrences of $\I$), is block column,  the remaining two are not.

To understand the structure of these Schubert problems, 
let $\ulambda=(\umu,\unu)$ be a block column Schubert problem as in Definition~\ref{D:BCSP}.
Observe that if we place the two rectangles in $\umu$ in opposite corners of the rectangle $\Box_{a,b}$ as $\Box_{a_1,b_1}$
and $(\Box_{a_2,b_2})^\circ$, then they lie along the main diagonal, meeting at their corners.
There are two remaining rectangles  $\Box_{a_2,b_1}$ and $\Box_{a_1,b_2}$ along the antidiagonal.
As $m_1+\dotsb+m_p=b_1$, the first block $(\Box_{a_2,m_1},\dotsc,\Box_{a_2,m_p})$ of $\unu$ fills the first rectangle, with
each partition spanning all $a_2$ rows, and the second block $(\Box_{a_1,n_1},\dotsc,\Box_{a_1,n_q})$ of $\unu$ similarly
fills the second rectangle.
We show this for the block column Schubert problems given above.
In each, the two partitions in $\umu$ are shaded.
\[
  \raisebox{-8pt}{\includegraphics{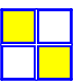}}\qquad
  \raisebox{-8pt}{\includegraphics{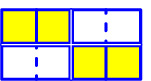}}\qquad
  \raisebox{-18pt}{\includegraphics{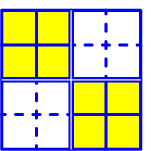}}\qquad
  \raisebox{-8pt}{\includegraphics{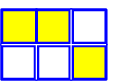}}\qquad
  \raisebox{-8pt}{\includegraphics{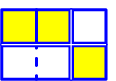}}\qquad
  \raisebox{-23pt}{\includegraphics{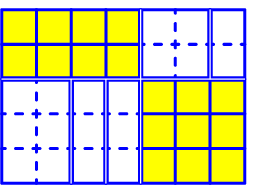}}
\]

%%%%%%%%%%%%%%%%%%%%%%%%%%%%%%%%%%%%%%%%%%%%%%%%%%%%%%%%%%%%%%%%%%%%%%%%%%%%%%%%%
\begin{lemma}\label{L:BCSP}
  Let $\ulambda=(\umu,\unu)$ be a block column Schubert problem on $G(a,b)$ with parameters 
  $a_1,a_2,b_1,b_2,m_1,\dotsc,m_p,n_1,\dotsc,n_q$ as in Definition~\ref{D:BCSP}.
  Let $\calFdot$ be $2+p+q$ general flags, which we write as
  $(\Edot^1,\Edot^2,\Fdot^1,\dotsc,\Fdot^p,\Gdot^1,\dotsc,\Gdot^q)$.
  The corresponding instance $\Omega_{\ulambda}\calFdot$ of $\ulambda$ consists of the set
\begin{multline}\label{Eq:BCSP}
   \{ H\in G(a,b)\ \mid\ \dim H\cap E^1_{a_1+b_2}\geq a_1\,,\ \dim H\cap E^2_{a_2+b_1}\geq a_2\,,\\
   \dim H\cap F^i_{b+a_2-m_i}\geq a_2\,,\
  \dim H\cap G^j_{b+a_1-n_j}\geq a_1\ \mbox{ for } i\in[p]\mbox{ and }j\in[q]\, \}\,.  
\end{multline}
\end{lemma}
%%%%%%%%%%%%%%%%%%%%%%%%%%%%%%%%%%%%%%%%%%%%%%%%%%%%%%%%%%%%%%%%%%%%%%%%%%%%%%%%%

%%%%%%%%%%%%%%%%%%%%%%%%%%%%%%%%%%%%%%%%%%%%%%%%%%%%%%%%%%%%%%%%%%%%%%%%%%%%%%%%%
\begin{proof}
  The partition $\Box_{c,d}$ has $c$ parts, each of size $d$.
  Thus if $\Fdot$ is a flag, then by~\eqref{Eq:SchubertVariety},
\[
    \Omega_{\Box_{c,d}}\Fdot\ =\
  \{ H\in G(a,b)\mid \dim H\cap F_{b+i-d}\geq i\ \mbox{ for }i=1,\dotsc,c\}\,.
\]
 The conditions on $H$ are implied by $\dim H\cap F_{b+c-d}\geq c$.
 The description~\eqref{Eq:BCSP} follows.
\end{proof}  
%%%%%%%%%%%%%%%%%%%%%%%%%%%%%%%%%%%%%%%%%%%%%%%%%%%%%%%%%%%%%%%%%%%%%%%%%%%%%%%%%

%%%%%%%%%%%%%%%%%%%%%%%%%%%%%%%%%%%%%%%%%%%%%%%%%%%%%%%%%%%%%%%%%%%%%%%%%%%%%%%%%
\begin{theorem}\label{Th:Fibration}
  Let $\ulambda$ be a block column Schubert problem on $G(a,b)$ and $\urho$ a Schubert problem on
  $G(c,d)$.
  Then the composed Schubert problem $\ulambda\circ\urho$ on $G(a{+}c,b{+}d)$ is fibered over $\ulambda$ with fiber
  $\urho$. 
\end{theorem}  
%%%%%%%%%%%%%%%%%%%%%%%%%%%%%%%%%%%%%%%%%%%%%%%%%%%%%%%%%%%%%%%%%%%%%%%%%%%%%%%%%

%%%%%%%%%%%%%%%%%%%%%%%%%%%%%%%%%%%%%%%%%%%%%%%%%%%%%%%%%%%%%%%%%%%%%%%%%%%%%%%%%
\begin{corollary}\label{C:Fibration}
 Suppose that $\ulambda$ is a nontrivial block column Schubert problem on $G(a,b)$ and $\urho$ is a nontrivial Schubert
 problem on $G(c,d)$.
 Then the Galois group $\Gal_{\ulambda\circ\urho}$ is imprimitive.  
\end{corollary}
%%%%%%%%%%%%%%%%%%%%%%%%%%%%%%%%%%%%%%%%%%%%%%%%%%%%%%%%%%%%%%%%%%%%%%%%%%%%%%%%%

%%%%%%%%%%%%%%%%%%%%%%%%%%%%%%%%%%%%%%%%%%%%%%%%%%%%%%%%%%%%%%%%%%%%%%%%%%%%%%%%%
\begin{remark}\label{R:BRSP}
  Conjugating every partition in a block column Schubert problem gives a \demph{block row Schubert problem}.
  A block row Schubert problem is composable, and the conclusions of Theorem~\ref{Th:Fibration} and
  Corollary~\ref{C:Fibration} hold for block row Schubert problems.
  The reason for this is that the isomorphism between $G(a,b)$ and $G(b,a)$ induced by duality corresponds to
  conjugation $\lambda\mapsto \lambda'$ of Young diagrams.
\hfill$\diamond$
\end{remark}
%%%%%%%%%%%%%%%%%%%%%%%%%%%%%%%%%%%%%%%%%%%%%%%%%%%%%%%%%%%%%%%%%%%%%%%%%%%%%%%%%

%%%%%%%%%%%%%%%%%%%%%%%%%%%%%%%%%%%%%%%%%%%%%%%%%%%%%%%%%%%%%%%%%%%%%%%%%%%%%%%%%
\subsection{Proof of Theorem~\ref{Th:Fibration}}
Let $\ulambda=(\umu,\unu)$ be a block column Schubert problem on $G(a,b)$ with the parameters
$a_1,a_2,b_1,b_2,m_1,\dotsc,m_p,n_1,\dotsc,n_q$ as in Definition~\ref{D:BCSP}.
Let $\urho=(\ualpha,\ubeta,\ugamma)$ be a Schubert problem on $G(c,d)$ with parameters $r=2$, $s=p{+}q$, and $t$ that we may
compose with $\ulambda$.
Then the composition $\ulambda\circ\urho$ consists of the partitions
\begin{multline*}
 \qquad \bigl( \Box_{a_1,(d+b_1)}, \alpha^1\bigr)\,,\, 
  \bigl( \Box_{a_2,(d+b_2)}, \alpha^2\bigr)\,,\
  \Box_{(c+a_2), m_1}{+}\beta^1\,,\,\dotsc\,,\, \Box_{(c+a_2), m_p}{+}\beta^p\;,\;\\
  \Box_{(c+a_1), n_1}{+}\beta^{p+1}\,,\,\dotsc\,,\, \Box_{(c+a_1), n_q}{+}\beta^{p+q}\ ,\
  \gamma^1,\dotsc,\gamma^t\,.\qquad
\end{multline*}
Here is a schematic illustrating the partitions in the first three groups
\begin{equation}\label{eq:ComposedSchematic}
  \raisebox{-55pt}{
  \begin{picture}(100,65)
      \put(13,15){\includegraphics{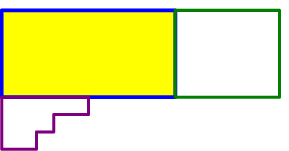}}
      \put(35,59){$d$}   \put(75,59){$b_i$}
      \put(0,42){$a_i$}
      \put(29,12){$\alpha^i$}
      \put(53,6){$i=1,2$}
    \end{picture}
  \qquad
  \begin{picture}(110,80)(1,0)
      \put(15,0){\includegraphics{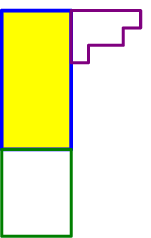}}
      \put(4,43){$c$}   \put(19,71){$m_j$}
      \put(1,10){$a_2$}
      \put(45,45){$\beta^j$}
      \put(45,2){$j=1,\dotsc,p$}
    \end{picture}
  \qquad
  \begin{picture}(100,80)(1,0)
      \put(15,0){\includegraphics{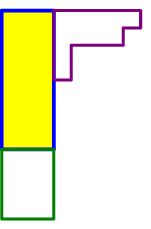}}
      \put(4,43){$c$}   \put(17,70){$n_k$}
      \put(1,12.5){$a_1$}
      \put(42,43){$\beta^{p+k}$}
      \put(40,6){$k=1,\dotsc,q$}
    \end{picture}}
 \end{equation}
%%%%%%%%%%%%%%%%%%%%%%%%%%%%%%%%%%%%%%%%%%%%%%%%%%%%%%%%%%%%%%%%%%%%%%%%%%%%%%%%%
 
Let $\calFdot=(\Fdot^1,\Fdot^2,\Kdot^1,\dotsc,\Kdot^p,\Kdot^{p+1},\dotsc,\Kdot^{p+q},\Ldot^1,\dotsc,\Ldot^t)$ be 
general flags in $\CC^{a+b+c+d}$.
We will use that they are general without comment in making assertions about the dimensions of intersections and spans.
Generality also implies that the dimension inequalities in the definition of Schubert variety~\eqref{Eq:SchubertVariety}
all hold with equality, which we also use.

By Definition~\ref{D:fibration}, to show that $\ulambda\circ\urho$ is fibered over $\ulambda$, we must do the following.

\begin{enumerate}
\item Construct a linear subspace $V\simeq\CC^{a+b}$ and an instance $\calEdot$ in $V$ of $\ulambda$ such that
       if $H\in\Omega_{\ulambda\circ\urho}\calFdot$, then $h:=H\cap V$ is an element of $\Omega_{\ulambda}\calEdot$.

\item  Let $W:=\CC^{a+c+b+d}/V$.
   Then,  for any $h\in \Omega_{\ulambda}\calEdot$, construct an instance
   $\calFdot(h)$ of $\urho$ in $\Gr(c,W)$ such that if $\defcolor{h}:=H\cap V$, then
   $H/h\in\Omega_{\urho}\calFdot(h)$.

 \item  Prove that the map $H\mapsto (h,H/h)$, where $h:=H\cap V$, is a bijection between 
   $\Omega_{\ukappa}\calFdot$ and $\Omega_{\ulambda}\calEdot\times \Omega_{\urho}\calFdot(h)$ (actually this is a
   fiber bundle over $\Omega_{\ulambda}$.)

 \item Show that the instances $\calEdot$ and $\calFdot(h)$ are sufficiently general.

\end{enumerate}

We prove Steps (1)--(4) in separate headings below.

%%%%%%%%%%%%%%%%%%%%%%%%%%%%%%%%%%%%%%%%%%%%%%%%%%%%%%%%%%%%%%%%%%%%%%%%%%%%%%%%%
% Copied from GIVIX
\begin{remark}
 A (partial) flag $\Fdot$ in $\CC^n$ is a nested sequence of subspaces, where not all dimensions need occur.
 For any subspace $h\subset\CC^n$, the sequence of subspaces $h+F_i$ for $F_i\in\Fdot$ forms another flag
 \defcolor{$h{+}\Fdot$} in $\CC^n$ with smallest subspace $h$.
 If $V\subset\CC^n$ is a subspace, then the subspaces $V\cap F_i$ form a flag \defcolor{$V\cap\Fdot$} in $V$.
 If $\CC^n\twoheadrightarrow V$ is surjective, then the image of $\Fdot$ is a flag in $V$.\hfill{$\diamond$}
%
%   This may be generalized further
%
\end{remark}
%%%%%%%%%%%%%%%%%%%%%%%%%%%%%%%%%%%%%%%%%%%%%%%%%%%%%%%%%%%%%%%%%%%%%%%%%%%%%%%%%

\noindent{\bf Step 1.}
Set $\defcolor{V}:= F^1_{b_2+a_1}+ F^2_{b_1+a_2}$.
As the flags are general, this is a direct sum and $V$ has dimension $a_1{+}a_2{+}b_1{+}b_2=a{+}b$.
Define \defcolor{$\calEdot$} by $\Edot^i := \Fdot^i\cap V$ for $i=1,2$ and 
$\Edot^{2+j} := \Kdot^j\cap V$ for $j=1,\dotsc,p{+}q$.
Fixing $F^1_{b_2+a_1}$ and $F^2_{b_1+a_2}$ and thus $V$, every possible collection of flags $\calEdot$ in $V$ can occur,
which proves part of Assertion (4) from Definition~\ref{D:fibration}. 

Let $H\in\Omega_{\ulambda\circ\urho}\calFdot$.
Since $H\in\Omega_{\bigl( \Box_{a_i,(d+b_i)}, \alpha^1\bigr)}\Fdot^i$, for $i=1,2$, it satisfies
\[
    \dim H\cap F^1_{d+b-(d+b_1)+a_1}\ =\ a_1\ \mbox{ and }\ 
    \dim H\cap F^2_{d+b-(d+b_2)+a_2}\ =\ a_2\,.
\]
The subspaces here are just $F_{b_2+a_1}$ and $F_{b_1+a_2}$.
Thus the dimension of $\defcolor{h}:=H\cap V$ is $a=a_1{+}a_2$.
This implies that $h\in\Gr(a,V)$, and that $h\in\Omega_{\Box_{a_i,b_i}}\Edot^i$ for $i=1,2$.

To show that $h\in\Omega_{\ulambda}\calEdot$, first let $1\leq j\leq p$.
Since $H\in\Omega_{ \Box_{(c+a_2), m_j}{+}\beta^j}\Kdot^j$, we have that 
\[
  \dim H\cap K^j_{b+d-m_j+c+a_2}\ =\ c+a_2\,.
\]
Since $V$ has codimension $c{+}d$, for any $i$ we have $\dim V\cap  K^j_{c+d+i}=i$.
Thus $E^j_{b-m_j+a_2}=V\cap K^j_{b+d-m_j+c+a_2}$.
Since $h$ has codimension $c$ in $H$, $\dim h\cap K^j_{b+d-m_j+c+a_2}\ =\ a_2$.
Putting these dimension  calculations together shows that 
\[
  \dim h\cap E^j_{b-m_j+a_2}\ =\ a_2\,,
\]
and thus $h\in\Omega_{\Box_{a_2,m_j}}\Edot^j$.
Similar arguments for $k=1,\dotsc,q$ show that $h\in\Omega_{\Box_{a_1,n_k}}\Edot^{p+k}$, and thus
$h\in\Omega_{\ulambda}\calEdot$.
This completes the proof of Assertion (1) in Definition~\ref{D:fibration}.\medskip

%%%%%%%%%%%%%%%%%%%%%%%%%%%%%%%%%%%%%%%%%%%%%%%%%%
\noindent{\bf Step 2.}
Let $h\in\Omega_{\ulambda}\calEdot$.
Since $\dim h\cap K^j_{b+d+c+a_2-m_j}=a_2$ and $\dim h=a$, we have that
\[
    \dim \bigl(h +  K^j_{b+d+c+a_2-m_j}\bigr)\ =\ a+b+c+d-m_j\,,
\]
That is, it has codimension $m_j$.
Similarly $h{+}K^{p+k}_{b+d+c+a_1-n_k}$ has codimension $n_k$.
Since the sum of the $m_j$ and of the $n_k$ is $b$, the intersection \defcolor{$W(h)$} of these spaces,
 \begin{equation}\label{Eq:W(h)}
     \Bigl(\bigcap_{j=1}^p \bigl(h +  K^j_{b+d+c+a_2-m_j}\bigr)\Bigr)\ \cap\ 
     \Bigl(\bigcap_{k=1}^q \bigl(h +  K^{p+k}_{b+d+c+a_1-n_k}\bigr)\Bigr)\ ,
 \end{equation}
has codimension $b$.   
Note that $W(h)$ contains $h$ and $W(h)/h\simeq W=\CC^{a+b+c+d}/V$.
For any flag $\Fdot$ in $\CC^{a+b+c+d}$, let $\Fdot(h)$ be the image of the flag $h+\Fdot$ in $W$.
We claim that $H\in\Omega_{\ulambda\circ\urho}\calFdot$, and $h=H\cap V$, then
$H/h\in \Omega_{\urho}\calFdot(h)$ in $\Gr(c,W(h)/h)$.

We show this for the first two conditions $\ualpha=(\alpha^1,\alpha^2)$ in $\urho$.
The first two conditions $(\Box_{a_i,d+b_i},\alpha^i)$ of $\ulambda\circ\urho$ are depicted on the left
of~\eqref{eq:ComposedSchematic}.
As $H\in\Omega_{(\Box_{a_1,d+b_1},\alpha^1)}\Fdot^1$, we have, 
\[
  \dim H\cap F^1_{b_2+a_1}\ =\ a_1
  \quad\mbox{ and }\quad
  \dim H\cap F^1_{d+b+a_1+j-\alpha^1_j}\ =\ a_1+j\ \mbox{ for } j=1,\dotsc,c\,.
\]
(The first is from Step 1.)
Since $H\cap F^1_{b_2+a_1}= h\cap F^1_{b_2+a_1}$, we have 
\[
  \dim H\cap (h+F^1_{d+b+a_1+j-\alpha^1_j})\ =\ a+j\ \mbox{ for } j=1,\dotsc,c\,.
\]
Then $\dim(h+F^1_{d+b+a_1+j-\alpha^1_j})=d+b+a+j-\alpha^1_j$, so that its image in $W(h)/h$ has dimension
$d+j-\alpha^1_j$, and thus
\[
    \dim \bigl(H/h \cap  F^1(h)_{d+j-\alpha^1_j}\bigr)\ =\ j\ \mbox{ for } j=1,\dotsc,c\,,
\]
which shows that $H/h\in \Omega_{\alpha^1}\Fdot^1(h)$.
The same arguments show that  $H/h\in \Omega_{\alpha^2}\Fdot^2(h)$.

We now consider the next $p+q$ conditions $\ubeta$ in $\urho$.
Suppose that $1\leq j\leq p$.
As $H\in\Omega_{\Box_{c+a_2,m_j}+\beta^j}\Kdot^j$, for each $i=1,\dotsc,c$, we have
\[
  \dim H\cap K^j_{b+d+i-m_j-\beta^j_i}\ =\ i\,.
\]
Since $V$ is in general position with respect to $\Kdot^j$, $h\cap K^j_{b+d+i-m_j-\beta^j_i}=\{0\}$, so that
\[
  \dim \bigl(h+   K^j_{b+d+i-m_j-\beta^j_i}\bigr)\ =\  a+b+d+i-m_j-\beta^j_i\,.
\]
Since this is a subspace of $h+ K^j_{b+d+c+a_2-m_j}$, which has codimension $m_j$ and is one of the subspaces in the
intersection~\eqref{Eq:W(h)} defining $W(h)$, we see that
\[
  \dim \Bigl(W(h)\;\cap\; \bigl(  h+   K^j_{b+d+i-m_j-\beta^j_i} \bigr)\Bigr) \ =\ a+d+i-\beta^j_i\,.
\]
Thus the image of $ h+   K^j_{b+d+i-m_j-\beta^j_i}$ in $W(h)/h$ has dimension $d+i-\beta^j_i$, so that it is the
subspace $K^j(h)_{d+i-\beta^j_i}$.
We conclude that for   $i=1,\dotsc,c$,
\[
  \dim \bigl(H/h\cap  K^j(h)_{d+i-\beta^j_i}\bigr)\ =\ i\,.
\]
Thus for $j=1,\dotsc,p$, we have $H/h\in\Omega_{\beta^j}\Kdot^j(h)$.
Similar arguments prove that $H/h\in\Omega_{\beta^{j}}\Kdot^{j}(h)$ for $j=p+1,\dotsc,p+q$.

To complete this step, we show that $H/h\in\Omega_{\gamma^j}\Ldot^{j}(h)$ for each $j=1,\dotsc,t$.
Let $1\leq j\leq t$.
Then $H\in\Omega_{\gamma^j}\Ldot^j$ and for each  $i=1,\dotsc,c$, we have
$\dim H\cap L^j_{b+d+i-\gamma^j_i}=i$.
Again by the general position of $V$, we have that
$\dim (h+ L^j_{b+d+i-\gamma^j_i})=a+b+d+i-\gamma^j_i$.
Then the intersection of this subspace with $W(h)$ has dimension $a+d+i-\gamma^j_i$, so that its
image in $W(h)/h$ has dimension $d+i-\gamma^j_i$, and is the subspace $L^j(h)_{d+i-\gamma^j_i}$ in the flag $\Ldot^j(h)$.
Almost as before, this implies that 
\[
  \dim \bigl(H/h\cap  L^j(h)_{d+i-\gamma^j_i}\bigr)\ =\ i\ \ \mbox{ for } i=1,\dotsc,\, c.
\]
Thus we conclude that for $j=1,\dotsc,t$, we have that $H/h\in\Omega_{\gamma^j}\Ldot^j(h)$.
To complete the proof of Assertion (2) of  Definition~\ref{D:fibration}, we only need to identify
$W(h)/h$ with $W$.\medskip

%%%%%%%%%%%%%%%%%%%%%%%%%%%%%%%%%%%%%%%%%%%%%%%%%%
\noindent{\bf Step 3.}
The constructions in Steps 1 and 2 give us, for every $H\in\Omega_{\ulambda\circ\urho}\calFdot$ a pair of subspaces
$h:=H\cap V$ and $H/h\in W(h)/h$, as well as instances $\calEdot$ of $\ulambda$ in $\Gr(a,V)$ and
$\calFdot(h)$ of $\urho$ in $\Gr(c,W(h)/h)$ such that $h\in\Omega_{\ulambda}\calEdot$ and
$H/h\in \Omega_{\urho}\calFdot(h)$.

The mapping $H\mapsto (h,H/h)$ is clearly injective.
It is surjective, as by Theorem~\ref{Th:product}, there are
$\delta(\ulambda\circ\urho)=\delta(\ulambda)\cdot\delta(\urho)$ elements $H$ in $\Omega_{\ulambda\circ\urho}\calFdot$,
which is the number of pairs $(h,H/h)$ such that  $h\in\Omega_{\ulambda}\calEdot$ and
$H/h\in \Omega_{\urho}\calFdot(h)$.
This completes Assertion (3) of  Definition~\ref{D:fibration}.\medskip

%%%%%%%%%%%%%%%%%%%%%%%%%%%%%%%%%%%%%%%%%%%%%%%%%%
\noindent{\bf Step 4.}
We already observed that $\calEdot$ is sufficiently general.
For the same assertions regarding $\calFdot(h)$, we observe that the flags $\Ldot^j$ are in general position with respect
to $V$, as are the subspaces $K^j_r$ for $r\leq b+d+c-m_j$, and finally
the subspaces $F^1_{b_2+a_1+r}/F^1_{b_2+a_1}$ for $r\geq 0$ are general in $W$, and the same for
$F^2_{b_1+a_2+r}/F^2_{b_1+a_2}$.
This completes the proof of Theorem~\ref{Th:Fibration}, as well as Corollary~\ref{C:Fibration}.\hfill$\Box$

%%%%%%%%%%%%%%%%%%%%%%%%%%%%%%%%%%%%%%%%%%%%%%%%%%%%%%%%%%%%%%%%%%%%%%%%%%%%%%%%%
\subsection{Computational evidence for Conjecture~\ref{C:imprimitive}}\label{S:computation}
We present evidence that nontrivial compositions $\ulambda\circ\urho$ of Schubert problems have imprimitive Galois groups,
even when $\ulambda$ is not a block column Schubert problem.

Let $\ulambda=(\umu,\unu)$ with  $\umu=(\T,\T,\I)$, $\nu=(\II,\I,\I)$, and $\ua=\ub=(1,1,1)$ be the composable partition
of Example~\ref{Ex:some_compositions}.
Let $\urho$ be the Schubert problem $(\I,\I,\I,\I)$ on $G(2,2)$, expanded with some empty
partitions, so that $\ualpha=\ubeta=(\I,\emptyset,\emptyset)$ and $\ugamma=(\I,\I)$.
The composition $\ulambda\circ\urho$ consists of the eight partitions
 \[
  \raisebox{-6pt}{\,\includegraphics{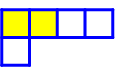}\,}\,,\,   
  \raisebox{-2pt}{\,\includegraphics{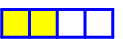}\,}\,,\,   
  \raisebox{-2pt}{\,\includegraphics{pictures/3d}\,}\;\ ,\  \;  
  \raisebox{-14pt}{\,\includegraphics{pictures/2111c}\,}\,,\,   
  \raisebox{-10pt}{\,\includegraphics{pictures/111c}\,}\,,\,   
  \raisebox{-10pt}{\,\includegraphics{pictures/111c}\,}\; \ ,\; \    
  \raisebox{-2pt}{\Ib}\, ,\,
  \raisebox{-2pt}{\Ib}   \ .
\]
Algorithms in Schubert calculus imply that $\delta(\ulambda)=5$ and $\delta(\urho)=2$, and we know from previous
computations that both are full symmetric, $\Gal_{\ulambda}=S_5$ and $\Gal_{\urho}=S_2$.

As explained in~\cite[\S~2.2]{GIVIX}, we may use symbolic computation to study the Galois group of a Schubert problem
$\ukappa$ of moderate size by computing cycle types of Frobenius lifts.
This does not quite study $\Gal_{\ukappa}$, but rather  $\Gal_{\ukappa}(\QQ)$, the Galois group of the
branched cover $\Omega_{\ukappa}\to(\Fl(a+b))^r$, where we have restricted scalars to $\QQ$, instead of $\CC$.

We computed 1 million Frobenius lifts in $\Gal_{\ulambda\circ\urho}(\QQ)$, finding 24 cycle types.
This computation supports the conjecture that $\Gal_{\ulambda\circ\urho}(\QQ)$ is the expected wreath product
$(S_2)^5\rtimes S_5$, which is the hyperoctahedral group $B_5$.
This group has $2^5\cdot 5!=3840$ elements.
Table~\ref{Ta:test} shows the frequency of cycle types found.
Each row is labeled by the cycle type, with the second column recording the frequency.
%%%%%%%%%%%%%%%%%%%%%%%%%%%%%%%%%%%%%%%%%%%%%%%%%%%%%%%%%%%%%%%%%%%%%%%%%%%%%%%%%
\begin{table}[htb]
\caption{Frequency of cycle types of $\ulambda\circ\urho$.}
\label{Ta:test}
%%%%%%%%%%%%%%%%%%%%%%%%%%%%%%%%%%%%%%%%
 \[
  \begin{tabular}{|c|r|r|r|||c|r|r|r|} \hline 
  \multicolumn{8}{|c|}{Cycles found in 1000000 instances of the \raisebox{-6pt}{\rule{0pt}{10pt}}
             Schubert problem $\ulambda\circ\urho$}\\\hline 
  Type & Freq. & Fraction &actual &  Type & Freq. & Fraction &actual\\
  \hline\hline
  $(10)$ &  99816 &    383.293 & 384
         & $(1^2, 2^2, 4)$ &  46399 &    178.172 & 180   \rule{0pt}{13pt}\\\hline
  $(2, 8)$ &  62583 &    240.319 & 240
         & $(1^4, 2, 4)$ &  15367 &     59.009 &  60     \rule{0pt}{13pt}\\\hline
  $(1^2, 8)$ &  62347 &    239.412 & 240
         &  $(1^6, 4)$ &   5157 &     19.803 &  20       \rule{0pt}{13pt}\\\hline
  $(4, 6)$ &  41883 &    160.831 & 160
         & $(2^2, 3^2)$ &  62596 &    240.369 & 240      \rule{0pt}{13pt}\\\hline
  $(2^2, 6)$ &  62690 &    240.730 & 240
         &$(1^2, 2, 3^2)$ &  41447 &    159.156 & 160    \rule{0pt}{13pt}\\\hline
  $(1^2, 2, 6)$ &  41362 &    158.830 & 160
         & $(1^4, 3^2)$ &  20932 &     80.379 &  80      \rule{0pt}{13pt}\\\hline
  $(1^4, 6)$ &  20899 &     80.252 &  80
         & $(2^5)$ &  20960 &     80.486 &  81          \rule{0pt}{13pt}\\\hline
  $(5^2)$ & 100844 &    387.241 & 384
         & $(1^2, 2^4)$ &  32217 &    123.713 & 125     \rule{0pt}{13pt}\\\hline
  $(2, 4^2)$ &  78690 &    302.170 & 300
         & $(1^4, 2^3)$ &  17806 &     68.375 &  70      \rule{0pt}{13pt}\\\hline
  $(1^2, 4^2)$ &  78429 &    301.167 & 300
         &$(1^6, 2^2)$ &   7706 &     29.591 &  30       \rule{0pt}{13pt}\\\hline
  $(3^2, 4)$ &  42056 &    161.495 & 160
         &  $(1^8, 2)$ &   1261 &      4.842 &   5       \rule{0pt}{13pt}\\\hline
  $(2^3, 4)$ &  36331 &    139.511 & 140
         & $(1^{10})$ &    222 &      0.852 &   1     \rule{0pt}{13pt}\\\hline
\end{tabular}
\]
\end{table}
%%%%%%%%%%%%%%%%%%%%%%%%%%%%%%%%%%%%%%%%%%%%%%%%%%%%%%%%%%%%%%%%%%%%%%%%%%%%%%%%%
The empirical fraction of times the identity was obtained, $10^6/222\simeq 4504.5$, 
suggests that the order of the group $|\Gal_{\ulambda\circ\urho}(\QQ)|$ is a divisor of $10!$ near this number.
Taking into account the empirical evidence from the computations in~\cite{GIVIX} that the identity is under sampled, we
suppose that it is one of 
\[
   3600\,,\  3780\,,\  3840\,,\  4032\,,\  4050\,,\  4200\,,\  4320\,,\  4480\,.
%#for i from 3500 to 4500 do   if (trunc(10!/i)*i=10!) then print(i); end if: end do:
\]
Assuming that $|\Gal_{\ulambda\circ\urho}(\QQ)|=3840$, the third column gives the (normalized to 3840) fraction of times
that cycle type was observed.
The last column gives the number of elements in $B_5$ with the given cycle type in $S_{10}$.

%%%%%%%%%%%%%%%%%%%%%%%%%%%%%%%%%%%%%%%%%%%%%%%%%%%%%%%%%%%%%%%%%%%%%%%%%%%%%%%%%
\providecommand{\bysame}{\leavevmode\hbox to3em{\hrulefill}\thinspace}
\providecommand{\MR}{\relax\ifhmode\unskip\space\fi MR }
\providecommand{\MRhref}[2]{%
  \href{http://www.ams.org/mathscinet-getitem?mr=#1}{#2}
}
\providecommand{\href}[2]{#2}


\begin{thebibliography}{10}

\bibitem{AR}
C.~Am\'endola and J.~I. Rodriguez, \emph{Solving parameterized polynomial
  systems with decomposable projections}, 2016, {\tt arXiv:1612.08807}.

\bibitem{BSS}
G.~Benkart, F.~Sottile, and J.~Stroomer, \emph{Tableau switching: Algorithms
  and applications}, J. Comb. Theory Ser. A \textbf{76} (1996), 11--43.

\bibitem{BdCS}
C.~J. Brooks, A.~Mart\'\i n~del Campo, and F.~Sottile, \emph{Galois groups of
  {S}chubert problems of lines are at least alternating}, Trans. Amer. Math.
  Soc. \textbf{367} (2015), 4183--4206.

\bibitem{SDSS}
T.~Brysiewicz, J.~I. Rodriguez, F.~Sottile, and T.~Yahl, \emph{Solving
  decomposable sparse systems}, 2020, {\tt arXiv:2001.04228}.

\bibitem{DW}
Harm Derksen and Jerzy Weyman, \emph{The combinatorics of quiver
  representations}, Ann. Inst. Fourier (Grenoble) \textbf{61} (2011), no.~3,
  1061--1131.

\bibitem{Fulton}
William Fulton, \emph{Young tableaux}, London Mathematical Society Student
  Texts, vol.~35, Cambridge University Press, Cambridge, 1997.

\bibitem{Ha79}
J.~Harris, \emph{Galois groups of enumerative problems}, Duke Math.~J.
  \textbf{46} (1979), 685--724.

\bibitem{J1870}
C.~Jordan, \emph{Trait\'e des substitutions et des \'equations alg\'ebrique},
  Gauthier-Villars, Paris, 1870.

\bibitem{Kleiman}
S.~L. Kleiman, \emph{The transversality of a general translate}, Compositio
  Math. \textbf{28} (1974), 287--297.

\bibitem{LS09}
A.~Leykin and F.~Sottile, \emph{Galois groups of {S}chubert problems via
  homotopy computation}, Math. Comp. \textbf{78} (2009), no.~267, 1749--1765.

\bibitem{Macdonald}
I.~G. Macdonald, \emph{Symmetric functions and {H}all polynomials}, second ed.,
  Oxford Classic Texts in the Physical Sciences, The Clarendon Press, Oxford
  University Press, New York, 2015.

\bibitem{MSJ}
A.~Mart\'\i n~del Campo and F.~Sottile, \emph{Experimentation in the {S}chubert
  calculus}, Schubert Calculus, Osaka 2012 (H.~Naruse, T.~Ikeda, M.~Masuda, and
  T.~Tanisaki, eds.), Advanced Studies in Pure Mathematics, vol.~71,
  Mathematical Society of Japan, 2016, pp.~295--336.

\bibitem{GIVIX}
A.~Mart\'\i n~del Campo, F.~Sottile, and R.~Williams, \emph{Classification of
  {S}chubert {G}alois groups in {G}r(4,9)}, {\tt arXiv.org/1902.06809}, 2019.

\bibitem{PS}
Gian~Pietro Pirola and Enrico Schlesinger, \emph{Monodromy of projective
  curves}, J. Algebraic Geom. \textbf{14} (2005), no.~4, 623--642.

\bibitem{Sagan}
Bruce~E. Sagan, \emph{The symmetric group}, second ed., Graduate Texts in
  Mathematics, vol. 203, Springer-Verlag, New York, 2001.

\bibitem{SW_double}
F.~Sottile and J.~White, \emph{Double transitivity of {G}alois groups in
  {S}chubert calculus of {G}rassmannians}, Algebr. Geom. \textbf{2} (2015),
  no.~4, 422--445.

\bibitem{Stanley}
Richard~P. Stanley, \emph{Enumerative combinatorics. {V}ol. 2}, Cambridge
  Studies in Advanced Mathematics, vol.~62, Cambridge University Press,
  Cambridge, 1999.

\bibitem{Va06a}
R.~Vakil, \emph{A geometric {L}ittlewood-{R}ichardson rule}, Ann. of Math. (2)
  \textbf{164} (2006), no.~2, 371--421, Appendix A written with A. Knutson.

\bibitem{Va06b}
\bysame, \emph{Schubert induction}, Ann. of Math. (2) \textbf{164} (2006),
  no.~2, 489--512.

\bibitem{Wielandt}
Helmut Wielandt, \emph{Finite permutation groups}, Translated from the German
  by R. Bercov, Academic Press, New York-London, 1964.

\end{thebibliography}
\end{document}